\documentclass{amsart}

\usepackage[english]{babel}
\usepackage[utf8]{inputenc}
\usepackage{amsmath}
\usepackage{amssymb}
\usepackage{amsthm}
\usepackage{amssymb}
\usepackage{graphicx}
\usepackage[colorinlistoftodos]{todonotes}
\usepackage{hyperref}
\usepackage{comment}
\usepackage{fancyhdr}

\usepackage{thmtools}
\usepackage{thm-restate}

\newtheorem{theorem}{Theorem}[section]
\newtheorem{remark}[theorem]{Remark}

\newtheorem{proposition}[theorem]{Proposition}
\newtheorem{definition}{Definition}[section]

\DeclareMathOperator{\Aut}{Aut}

\newcommand{\Q}{\mathbb{Q}}
\newcommand{\Z}{\mathbb{Z}}

\newcommand{\id}{{\mathtt {id}}}
\newcommand{\x}{\underline{x}}
\newcommand{\y}{\underline{y}}
\newcommand{\z}{\underline{z}}
\newcommand{\uu}{\underline{u}}
\newcommand{\vv}{\underline{v}}
\newcommand{\ww}{\underline{w}}
\newcommand{\aaa}{\underline{a}}
\newcommand{\bb}{\underline{b}}
\newcommand{\cc}{\underline{c}}
\newcommand{\Y}{\underline{Y}}
\newcommand{\ZZ}{\underline{Z}}
\newcommand{\xo}{\underline{x_1}}
\newcommand{\xt}{\underline{x_2}}
\newcommand{\xth}{\underline{x_3}}
\newcommand{\yo}{\underline{y_1}}
\newcommand{\yt}{\underline{y_2}}
\newcommand{\yth}{\underline{y_3}}
\newcommand{\zo}{\underline{z_1}}
\newcommand{\zt}{\underline{z_2}}
\newcommand{\zth}{\underline{z_3}}
\newcommand{\uo}{\underline{u_1}}
\newcommand{\ut}{\underline{u_2}}
\newcommand{\uth}{\underline{u_3}}
\newcommand{\vo}{\underline{v_1}}
\newcommand{\vt}{\underline{v_2}}
\newcommand{\vth}{\underline{v_3}}
\newcommand{\wo}{\underline{w_1}}
\newcommand{\wt}{\underline{w_2}}
\newcommand{\wth}{\underline{w_3}}

\newcommand{\nn}{\nonumber}

\newcommand{\cube}[9]{ 
\begin{tikzpicture}[scale=0.9, baseline=(current bounding box.center)]

\node (a) at (0,2) {$#1$};
\node (b) at (2,2) {$#2$};
\node (c) at (0,0) {$#3$};
\node (d) at (2,0) {$#4$};
\node (e) at (0.8,2.8) {$#5$};
\node (f) at (2.8,2.8) {$#6$};
\node (g) at (0.8,0.8) {$#7$};
\node (h) at (2.8,0.8) {$#8$};

\draw (d) -- (c);
\draw (c) -- (a);
\draw (a) -- (e);
\draw (b) -- (f);
\draw (c) -- (g);
\draw (d) -- (h);
\draw (e) -- (f);
\draw (f) -- (h);
\draw (h) -- (g);
\draw (g) -- (e);
\draw[line width=10pt,white] (a) -- (b);
\draw (a) -- (b);
\draw[line width=10pt, white] (b) -- (d);
\draw (b) -- (d);

\node (p) at (3.2, 0.7) {#9}; 
\end{tikzpicture}
}
\makeatletter
\newcommand*\rel@kern[1]{\kern#1\dimexpr\macc@kerna}
\newcommand*\widebar[1]{%
  \begingroup
  \def\mathaccent##1##2{%
    \rel@kern{0.8}%
    \overline{\rel@kern{-0.8}\macc@nucleus\rel@kern{0.2}}%
    \rel@kern{-0.2}%
  }%
  \macc@depth\@ne
  \let\math@bgroup\@empty \let\math@egroup\macc@set@skewchar
  \mathsurround\z@ \frozen@everymath{\mathgroup\macc@group\relax}%
  \macc@set@skewchar\relax
  \let\mathaccentV\macc@nested@a
  \macc@nested@a\relax111{#1}%
  \endgroup
}
\makeatother

\title[Explicit composition identities for higher composition laws]{Explicit composition identities for higher composition laws}

\author{Gautam Chinta}
\address{Dept. of Mathematics, The City College of New York, New York, NY 10031 USA}
\email{gchinta@ccny.cuny.edu}

\author{Ajith Nair}
\address{School of Mathematical and Statistical Sciences, Arizona State University, Tempe, AZ 85287-1804, USA}
\email{ajith.nair@asu.edu}

\date{\today}

\begin{document}
\maketitle

\section*{Abstract} \label{Abs}

In 2001, Bhargava proved a composition law for $2 \times 2 \times 2$ integer cubes, which generalized Gauss composition of integral binary quadratic forms. Furthermore, he derived four new composition laws defined on the following spaces: 1) binary cubic forms with triplicate middle coefficients, 2) pairs of binary quadratic forms with duplicate middle coefficients, 3) pairs of quaternary alternating 2-forms and 4) senary alternating 3-forms. In each of the five cases, there is a natural group action on the underlying space with a unique polynomial invariant called the discriminant, and a notion of projectivity for the elements of the space. The strategy behind Bhargava's approach is to construct a discriminant-preserving bijection between the set of orbits under the group action and the set of (tuples of) suitable ideal classes of quadratic rings. The projective ideal classes are equipped with a natural group structure and hence we get a group structure on the spaces of equivalence classes of projective forms of fixed discriminant $D$. In each case the class group of projective forms of discriminant $D$ has a natural interpretation in terms of the narrow class group of the quadratic ring of discriminant $D$. The aim of this paper is to give explicit composition identities (similar to Gauss' formulation of composition of binary quadratic forms) for these higher composition laws.

\section{Introduction} \label{sec:Intro}

\subsection {Background}
Let $Q_1,Q_2$ and $Q_3$ be three integral binary quadratic forms of discriminant $D$.  Then, $Q_3$ is said to be a \emph{composition} of $Q_1$ and $Q_2$ if
\begin{equation}
  \label{eq:compositionF1F2F3}
Q_1(x_1,x_2) Q_2(y_1,y_2) = Q_3(z_1,z_2),
\end{equation}
where $z_1$ and $z_2$ are bilinear forms in $(x_1,x_2)$ and $(y_1,y_2)$, i.e.
        \begin{align}\label{eq:bqf_bilinear}
          z_1 &= a_{11}x_1 y_1 + a_{12} x_1 y_2 + a_{21} x_2 y_1 + a_{22} x_2 y_2\\ \nonumber
z_2 &= b_{11}x_1 y_1 + b_{12} x_1 y_2 + b_{21} x_2 y_1 + b_{22} x_2 y_2
        \end{align}
      for some $a_{ij}, b_{ij} \in \Z \ (i,j=1,2)$,
and
\begin{equation}
  \label{eq:bqfQs}
Q_1(1,0) = a_{11} b_{12} - a_{12} b_{11}, Q_2(1,0) = a_{11} b_{21} - a_{21} b_{11}.
\end{equation}
The simplest and oldest example is the Brahmagupta-Fibonacci-Diophantus identity:
\begin{equation*}
    (a^2+b^2)(c^2+d^2)=(ac+bd)^2+(ad-bc)^2,
\end{equation*}
expressing a product of a sum of two squares as again a sum of two squares.

Denote by $(\text{Sym}^2 \Z^2)^{*}$ the space of binary quadratic forms with integer coefficients.  The group $\text{SL}_2(\Z)$ acts on this space by linear substitution of the variables: a matrix $g = \begin{pmatrix} a & b \\ c & d \end{pmatrix} \in \text{SL}_2(\Z)$ acts on a binary quadratic form $Q(x,y)$ by
\begin{equation}
\label{eq:sl2action}
    Q \mapsto Q^g = Q(ax + by, cx + dy).
\end{equation}
Under this action the discriminant of $Q$ and the g.c.d. of its coefficients are preserved. We call the form $Q$ primitive if its coefficients are relatively prime.  Gauss (\cite{cf1801disquisitiones}) showed that the $\text{SL}_2(\Z)$-equivalence classes of primitive binary quadratic forms of a fixed discriminant $D$ have the structure of an abelian group under composition. We call this group the class group and denote it by $\text{Cl}((\text{Sym}^2 \Z^2)^{*};D)$. The identity element of the class group $\text{Cl}((\text{Sym}^2 \Z^2)^{*};D)$ is the class of the principal form of discriminant $D$ given by
\begin{equation}\label{eq:principal_form}
    Q_{\text{id}}(x,y) = \begin{cases}
        x^2 - \frac{D}{4}y^2 & \text{if } D \equiv 0 \text{ (mod } 4), \\
        x^2 + xy + \frac{1-D}{4}y^2 & \text{if } D \equiv 1 \text{ (mod } 4).
    \end{cases}
\end{equation}
Gauss' formulation has the disadvantage that it is not easy to compute the composition of two given forms. Dirichlet (\cite{dirichlet}) provides a simple method to compute the composition by taking suitable representatives equivalent to the given forms.

In 2001, Bhargava (\cite{bhargava:2004}) gives the following interpretation of Gauss composition using $2 \times 2 \times 2$ integer cubes. Consider a cube $A$ (as shown below) with eight integers $a,b,c,d,e,f,g,h$ at the vertices:
\begin{equation}
  \label{eq:cube}
A=\cube{a}{b}{c}{d}{e}{f}{g}{h}.
\end{equation}
By slicing the cube along the three possible planes, we get three pairs of $2 \times 2$ matrices, say $M_i$ and $N_i$ (for $i=1,2,3$). Let $Q_1^{A},Q_2^{A},Q_3^{A}$ be binary quadratic forms associated to $A$ given by
\begin{equation} \label{eq:CubeBQFs}
Q_i^{A}(x,y) := -\det(M_i x - N_i y).
\end{equation}
For example
\[Q_1^{A}(x,y) :=
-\det\left(
    \begin{pmatrix}
      a&b\\c&d
    \end{pmatrix}  x
-
    \begin{pmatrix}
      e&f\\ g&h
    \end{pmatrix}y
\right),
\]
and similarly for $Q_2^A, Q_3^A$.
See Section \ref{sec:cubecomp} for complete definitions. The three binary quadratic forms $Q_1^{A}, Q_2^{A}, Q_3^{A}$ have the same discriminant, which is also defined to be the discriminant of the cube $A$. Then, Bhargava imposes the relation on the space of primitive binary quadratic forms of a fixed discriminant $D$, that for any triple of primitive forms $Q_1^A, Q_2^A, Q_3^A$ arising from a cube $A$, their formal sum is zero.  This is the Cube Law of Bhargava \cite{bhargava:2004}.  This relation descends to a relation on the space of $\text{SL}_2(\Z)$-classes of primitive forms of discriminant $D$, and after an identity element is chosen, this space becomes a finite abelian group. Moreover, with the choice of identity as the principal binary quadratic form of discriminant $D$ given in \eqref{eq:principal_form}, this group is the same as the class group $\text{Cl}((\text{Sym}^2 \Z^2)^{*};D)$.

In fact, as seen for example in Lemmermeyer (\cite{lemmermeyer2011binary}, Theorem 2), the cube \eqref{eq:cube} also gives us the bilinear forms \eqref{eq:bqf_bilinear} in the composition identity of the forms associated to the cube. Precisely, if $Q^A_1,Q^A_2,Q^A_3$ are the binary quadratic forms \eqref{eq:CubeBQFs} associated to a cube $A$, then
\begin{equation} \label{eq:lemmermeyerid}
    Q^A_2(x_2,-y_2)Q^A_3(x_3,-y_3) = Q^A_1(x_1,y_1),
\end{equation}
where 
\begin{equation*}
        x_1 = \begin{pmatrix} x_2 & y_2 \end{pmatrix} \begin{pmatrix} e & f \\ g & h \end{pmatrix} \begin{pmatrix} x_3 \\ y_3 \end{pmatrix},
        y_1 = \begin{pmatrix} x_2 & y_2 \end{pmatrix} \begin{pmatrix} a & b \\ c & d \end{pmatrix} \begin{pmatrix} x_3 \\ y_3 \end{pmatrix}.
\end{equation*}
The above composition identity is readily verified by substitution. (Note that the definitions of the quadratic forms $Q^A_1,Q^A_2,Q^A_3$ are slightly different in \cite{lemmermeyer2011binary}). There are two more such identities for the composition of $Q_1^A$ with $Q_2^A$, and $Q_1^A$ with $Q_3^A$. (See Proposition 2.2.1 in \cite{nair2024explicit}.)

\subsection{The Cube Law and dual cubes}
\label{subsec:CubeLaw}
The significance of Bhargava's approach comes not only from its reinterpretation of Gauss composition, but also from the generalizations of the notion of composition to higher degree forms. We will describe the explicit higher composition law for cubes here in the introduction and leave the detailed statements of the other cases to subsequent sections.

The space $\Z^2 \otimes \Z^2 \otimes \Z^2$, whose elements can be represented as $2 \times 2 \times 2$ cubes as shown above, has a natural action by the group $\Gamma = \text{SL}_2(\Z) \times \text{SL}_2(\Z) \times \text{SL}_2(\Z).$  The discriminant of the cube, (which is defined to be the common discriminant of the three binary quadratic forms associated to the cube), generates the ring of polynomial invariants of this action. A cube is said to be \emph{projective} if the three quadratic forms associated to it are primitive.

Bhargava shows that the space of $\Gamma$-equivalence classes of projective cubes of discriminant $D$ can also be equipped with the structure of a finite abelian group, which is denoted by $\text{Cl}(\Z^2 \otimes \Z^2 \otimes \Z^2;D)$. Let $[A]$ denote the $\Gamma$-equivalence class of a cube $A$.  The identity element in $\text{Cl}(\Z^2 \otimes \Z^2 \otimes \Z^2;D)$ is $[A_{\text{id},D}]$ for
\begin{equation}
  \label{def:id_cube}
  A_{\text{id},D}=\cube {0}{1}{1}{0}{1}{0}{0}{\frac D4}{}
  \mbox{ or }
  \cube {0}{1}{1}{1}{1}{1}{1}{\ \ \tfrac {D+3}4}{\ \ \ ,}
\end{equation}
as $D$ is 0 or 1 mod 4.
All three quadratic forms associated to the triply symmetric cube $A_{\text{id},D}$ (for $D \equiv 0 \text{ or } 1 \text{ (mod } 4)$) are equal to the principal binary quadratic form of discriminant $D$ given in \eqref{eq:principal_form}.  When the discriminant $D$ is clear from the context we write more simply $[\id]=[A_{\text{id}, D}].$
The group structure on the space of $\Gamma$-equivalence classes of $2 \times 2 \times 2$ cubes of discriminant $D$ follows from the following bijection:
\begin{equation*}
\begin{Bmatrix}
\Gamma \text{-orbits of elements of } \\ \Z^2 \otimes \Z^2 \otimes \Z^2
\end{Bmatrix}
\longleftrightarrow \begin{Bmatrix}
\text{isomorphism classes of pairs}\\ (S,(I_1,I_2,I_3))
\end{Bmatrix}.
\end{equation*}
Here $S$ is an oriented quadratic ring and $(I_1,I_2,I_3)$ is an equivalence class of balanced triple of oriented ideals of $S$. The discriminant of a cube equals the discriminant of the corresponding oriented quadratic ring under this bijection. See Section \ref{sec:cubecomp} for more details.

Alternatively, given three projective cubes $A_1,A_2,A_3$ of common discriminant $D$, we have
\begin{equation*}
  \label{eq:sum_cubes}
  [A_1]+[A_2]+[A_3] = [\id]
\end{equation*}
if and only if the sum
\begin{equation*}
  \label{eq:4}
[Q_i^{A_1}]+[Q_i^{A_2}]+[Q_i^{A_3}]
\end{equation*}
is the identity in $\text{Cl}((\text{Sym}^2 \Z^2)^{*};D)$ for each $i=1,2,3.$
From the ordered triple $(A_1,A_2,A_3)$ of projective cubes as above we get 9 primitive quadratic forms which we arrange in a 3-by-3 matrix
\begin{equation*}
  \label{eq:quadratic_forms}
\left[  \begin{array}{ccc}
Q_1^{A_1} & Q_2^{A_1} & Q_3^{A_1} \\
Q_1^{A_2} & Q_2^{A_2} & Q_3^{A_2} \\
Q_1^{A_3} & Q_2^{A_3} & Q_3^{A_3}
  \end{array}
\right].
\end{equation*}
By the Cube Law, each row sums to the identity in the class group of binary quadratic forms.  Moreover, since $[A_1]+[A_2]+[A_3]=[\id]$, each column sums to the identity as well.  This motivates our definition of {\em dual cubes} $(T_1, T_2, T_3)$ which are defined by the condition
\begin{equation}
  \label{def:dualcubes}
Q_i^{T_j}=Q_{j}^{A_i} \mbox{\ for \ } i,j=1,2,3.
\end{equation}
The existence of dual cubes follows from the converse part of Theorem 1 in (\cite{bhargava:2004}) which states that if $[Q_1] + [Q_2] + [Q_3]=[\id]$ in $\text{Cl}((\text{Sym}^2 \Z^2)^*;D)$, then there is a cube giving rise to $Q_1,Q_2$ and $Q_3$.

The cube $T_j$ is determined up to the action of $\Aut(Q_{j}^{A_1})\times \Aut(Q_j^{A_2})\times \Aut(Q_j^{A_3})$.  It is easy to see that the dual cubes are also each of discriminant $D$ and satisfy
$$[T_1]+[T_2]+[T_3]=[\id].$$  The dual cubes play a prominent rule in the explicit composition law for cubes as we now describe.

\subsection{Explicit composition of cubes}
Before stating our first main result we need a few more definitions.  Let the cube $A\in \Z^2\otimes\Z^2\otimes\Z^2$ have coefficients $(a_{ijk})$.  That is,
\begin{equation}
  \label{def:aijk}
A=\cube{a_{111}}{a_{112}}{a_{121}}{a_{122}}{a_{211}}{a_{212}}{a_{221}}{a_{222}}{.}
\end{equation}
We can think of $A$ as representing either a trilinear form or a pair of bilinear forms on $\Z^2$ as follows:
\begin{align}
  \label{def:cube_to_trilinear}
  A((x_1,x_2)^t,&(y_1, y_2)^t, (z_1, z_2)^t) =
\sum_{i,j,k=1}^2 a_{ijk}x_iy_jz_k
\end{align}
 and
\begin{align} \label{def:cube_to_bilinear}
A((y_1,y_2)^t,&(z_1, z_2)^t) = 
\begin{pmatrix}A(e_1, (y_1, y_2)^t, (z_1, z_2)^t)\\
A(e_2, (y_1, y_2)^t, (z_1, z_2)^t)
\end{pmatrix}
\end{align}
where $e_1, e_2$ are the standard basis vectors.

We further define three cubes whose coefficients are expressed in terms of the coefficients of $A$.  The involution $A\mapsto A^\iota$ swaps the front and back faces 
\begin{equation*}
  A^\iota=\cube{a_{211}}{a_{212}}{a_{221}}{a_{222}}{a_{111}}{a_{112}}
{a_{121}}{a_{122}} \ ,
\end{equation*}
and $A^\sigma$ is obtained by negating the front face of $A^\iota$
\begin{equation*}
  A^\sigma=\cube{(-a_{211})}{(-a_{212})}{(-a_{221})}{(-a_{222})}{a_{111}}{a_{112}}
{a_{121}}{a_{122}} \ .
\end{equation*}
Next, if $Q_1^A(x,y)=px^2+qxy+ry^2,$  the \emph{companion cube} $A'$ of $A$ is defined by
\begin{equation*}
A'(\x,\y,\z)=A(\x'-\epsilon \x,\y,\z)\mbox{ for } \x,\y,\z\in\Z^2,
\end{equation*}
where $\epsilon\in\{0,1\}$ is congruent to $D$ mod 4 and
$$\x'=
\begin{pmatrix}
  \tfrac {q+\epsilon}2& -r\\ p & \tfrac{-q+\epsilon}2
\end{pmatrix}\x.
$$
The coefficients of the companion cube $A'$ are the $c_{ijk}$ of equation (15) in \cite{bhargava:2004}. See Proposition \ref{prop:companioncube}.

Finally, define the \emph{form product} of two cubes $A,B\in \Z^2\otimes\Z^2\otimes\Z^2$ of common discriminant $D$ by
\begin{align}
  \label{def:form_product_cubes}
  (A\star B)&(\x,\y,\z;\uu,\vv,\ww)\\ \nn
=&A(\x,\y,\z)B'(\uu,\vv,\ww)
+A'(\x,\y,\z)B(\uu,\vv,\ww)
+\epsilon A(\x,\y,\z)B(\uu,\vv,\ww)
\end{align}
for $\x,\y,\z,\uu,\vv,\ww\in \Z^2.$
We can now state our first main result.


\begin{restatable*}[]{thm}{cubecomp}
\label{thm:cubecomp}
  Let $A, B, C\in \Z^2\otimes\Z^2\otimes\Z^2$ be three projective cubes of common discriminant $D$ and $\epsilon\in\{0,1\}$ congruent to $D$ mod 4. Let $\x,\y,\z,\uu,\vv,\ww\in \Z^2.$ Then
$$[A]+[B]+[C]=[\id]$$
if and only if there exist cubes $R=(r_{ijk}),S=(s_{ijk}),T=(t_{ijk})$ of discriminant $D$ satisfying
\begin{align} \label{eq:cubecomp_identity}
(B\star C)&(\x,\y,\z;\uu,\vv,\ww)\\ \nn
&= A(R^\sigma(\x,\uu), S^\sigma(\y,\vv), T^\sigma(\z,\ww))
\end{align}
and 
    \begin{align}\label{eq:cubecompQ}
        Q_1^{R} = Q_1^A, \ Q_2^{R} = Q_1^B,
    \end{align}
and
\begin{align}\label{eq:cubecompQ2}
  Q_1^B(1,0)\cdot Q_1^C(1,0)&=Q_1^A(r_{211},r_{111}),\\ \nn
  Q_2^B(1,0)\cdot Q_2^C(1,0)&=Q_2^A(s_{211},s_{111}),\\ \nn
  Q_3^B(1,0)\cdot Q_3^C(1,0)&=Q_3^A(t_{211},t_{111}).
\end{align}
If this is the case, the triple of cubes $(R, S, T)$ is dual to $(A,B,C)$.
\end{restatable*}
We prove this result and give an example of the composition identity in Section \ref{sec:cubecomp}.
\begin{remark}
    Note that we can write the composition identity in the above theorem in an alternative manner by choosing the representative
    \begin{equation}
      \label{def:Atilde}
      \tilde{A}=
\cube{-a_{111}}{a_{112}}{a_{121}}{(-a_{122})}{a_{211}}{(-a_{212})}{(-a_{221})}
{a_{222}}{.}
    \end{equation}
in the inverse class of $[A]$. The identity we get in this case is
    $$ (B \star C) (\underline{x},\underline{y},\underline{z};\underline{u},\underline{v},\underline{w}) = \tilde{A}(R^\iota(\underline{x},\underline{u}),S^\iota(\underline{y},\underline{v}),T^\iota(\underline{z},\underline{w})).$$
\end{remark}

\subsection{Summary of main results and outline of paper}
The full list of spaces with higher composition laws that Bhargava considers in
 \cite{bhargava:2004} is
\begin{enumerate}
\item $\Z^2 \otimes \Z^2 \otimes \Z^2$: binary trilinear forms,
\item $\text{Sym}^3 \Z^2$: binary cubic forms,
\item $\Z^2 \otimes \text{Sym}^2 \Z^2$: pairs of binary quadratic forms,
\item $\Z^2 \otimes \wedge^2 \Z^4$: pairs of quaternary alternating $2$-forms, and
\item $\wedge^3 \Z^6$: senary alternating $3$-forms.
\end{enumerate}
The spaces in (2)-(5) are obtained from the space $\Z^2 \otimes \Z^2 \otimes \Z^2$ by using certain multilinear operations, such as symmetrization, skew-symmetrization and castling.
Each of the five spaces $V(\Z)$ considered in the higher composition laws above is endowed with a natural action by a group $G(\Z)$, and in each case, there is a unique polynomial invariant under this action, which is called the discriminant. This is in complete analogy with the case $(\text{Sym}^2 \Z^2)^*$  of binary quadratic forms. Also, there is a notion of projectivity for elements of these spaces $V(\Z)$ which is analogous to primitivity for binary quadratic forms. Bhargava shows that in each of the five cases the set of projective $G(\Z)$-orbits of elements in $V(\Z)$ having discriminant $D$ has a natural structure of a finite abelian group.
The proof of the group structures associated to the higher composition laws relies on correspondences between the $G(\Z)$-orbits of elements in $V(\Z)$ and certain suitable ideal classes of quadratic rings. Establishing these correspondences constitutes a major part of the work in \cite{bhargava:2004}.

In this article, our goal is to formulate the five higher composition laws for the spaces above in a manner similar to Gauss's formulation of the composition law for binary quadratic forms, without mentioning the underlying ideal class group.  Theorem \ref{thm:cubecomp} above is our first example in the case of $\Z^2\otimes \Z^2\otimes\Z^2.$  The case of binary cubic forms $\text{Sym}^3 \Z^2$ has previously been considered by Nativi \cite{nativi2019analogue} whose results (after a simple correction) are essentially the same as ours in Section \ref{sec:bcfcomp}.  

Our results in all cases follow the pattern established in Eqs. (\ref{eq:compositionF1F2F3}-\ref{eq:bqfQs}) and Eq. \ref{eq:lemmermeyerid}: 
\begin{quote}\textit{
a product (or sum of products) of two forms is equated with a third form evaluated after a bilinear (or multilinear) change of variable; then the multilinear variable change is explicitly identified in terms of the two original forms.
}
\end{quote}

The spaces associated to the higher composition laws also appear in the study of prehomogeneous vector spaces over algebraically closed fields by Sato-Kimura \cite{satokimura}, and over the field of rational numbers by Wright-Yukie \cite{Wright1992}. Bhargava's work illustrates that the integer orbits of these spaces also have a rich and interesting structure.

Moreover, in \cite{hcl2}, \cite{hcl3} and \cite{hcl4}, Bhargava describes more composition laws (also originating in the work of \cite{satokimura} and \cite{Wright1992} on prehomogeneous vector spaces) where the corresponding spaces of forms parametrize arithmetic objects related to cubic, quartic and quintic rings. These parametrizations may be used to obtain results on counting number fields of degree $n\leq 5$ with bounded discriminant. See, for example \cite{bhargavaquartic, bhargava:quintic}.
There have also been some recent papers on composition identities with certain applications to Diophantine equations. Among them, we mention Choudhry \cite{choudhry:2022} and Duke \cite{duke2022some}. We also point out the thesis of M. Oller Riera \cite{oller2021bhargava} which builds upon the work of Bhargava in the cubic case \cite{hcl2} and the work of Bhargava-Ho \cite{bhargavaho} to derive composition laws of genus 1 curves.

The article is organized as follows. In Section \ref{sec:Gausscomp}, we review the relation between Gauss composition and multiplication of ideals in quadratic rings.  Sections \ref{sec:cubecomp}, \ref{sec:bcfcomp}, \ref{sec:pbqfcomp}, \ref{sec:pqaltcomp} and \ref{sec:senaltcomp} are devoted to our main task of deriving explicit composition identities for the five higher composition laws. In each of these sections, we start by explaining the correspondence between forms and ideals as described in \cite{bhargava:2004}, and then go on to state and prove the composition identities, which give an alternative description of the group law on these spaces. 
We also include examples of composition laws at the end of the corresponding sections.

\subsection*{Acknowledgments}
The authors are grateful to Fikreab Solomon for his help and advice in the early stages of this project.  The first named author is supported by the MPS-TSM program of the Simons Foundation and PSC-CUNY Award 65400-00-55.

\section{Gauss composition of binary quadratic forms} \label{sec:Gausscomp}

As a warm-up to the identities for the higher composition laws, we quickly review the classical case of Gauss composition of binary quadratic forms. We begin by describing the correspondence between binary quadratic forms and oriented ideals using the notation in \cite{bhargava:2004}. 
See also \cite{cohen}, \cite{cox} and \cite{bouyer} for further background.

A \emph{quadratic ring} is a commutative ring (with 1) whose additive group is a free $\Z$-module of rank 2. Let $S$ be a quadratic ring and $\langle \alpha_1, \alpha_2 \rangle$ a $\Z$-basis of $S$. Then, the \emph{discriminant} of $S$ is defined as $\text{det}(\text{tr}(\alpha_i \alpha_j))$ where $\text{tr}(x)$ is the trace of the ``multiplication by $x$'' endomorphism. The discriminant is an integer and is congruent to $0$ or $1$ (mod $4)$. Let $\epsilon \in \lbrace 0,1 \rbrace$ and $D\equiv\epsilon$ (mod $4$).  Up to isomorphism, there exists a unique quadratic ring $S(D)$ of discriminant $D$ with $\Z$-basis $\langle 1,\tau \rangle$ satisfying
\begin{equation} \label{eq:taurelation}
\tau^2 = \epsilon \tau + \frac{D - \epsilon}{4}.
\end{equation}
To avoid the extra automorphism of $S(D)$, which sends $\sqrt{D}$ to $-\sqrt{D}$, one considers the notion of an \emph{oriented quadratic ring}, which is a quadratic ring $S(D)$ together with a choice of $\sqrt{D}$. An oriented quadratic ring is said to be \emph{non-degenerate} if its discriminant is nonzero.

A $\Z$-basis $\langle 1 ,\tau \rangle$ of $S = S(D)$ is said to be positively oriented if $\text{tr}(\tau/\sqrt{D}) > 0$. Let $K = S \otimes \Q$. A $\Z$-basis $\langle \alpha_1,\beta_1,\alpha_2,\beta_2,...,\alpha_n,\beta_n \rangle$ of a rank $2n$ submodule of $K^n$ is said to be positively (or negatively) oriented if the change-of-basis matrix taking the $\Q$-basis 
$$\langle (1,0,...,0),(\tau,0,...,0),(0,1,...,0),(0,\tau,...,0),...,(0,...,1),(0,...,\tau) \rangle$$ 
of $K^n$ to the basis $\langle \alpha_1,\beta_1,\alpha_2,\beta_2,...,\alpha_n,\beta_n \rangle$ has positive (or negative) determinant.

An \emph{oriented ideal} of an oriented quadratic ring $S$ is a pair $(I,\mu)$, where $I$ is a fractional ideal of $S$ and $\mu = \pm 1$ is the orientation.  Multiplication of oriented ideals is defined component-wise. Two oriented ideals $(I_1,\mu_1)$ and $(I_2,\mu_2)$ are said to be equivalent if there exists $\kappa \in K^{\times}$ such that $(I_1,\mu_1) = (\kappa I_2, \text{sgn}(N(\kappa)) \mu_2) $. The norm of an oriented ideal $(I,\mu)$ is given by 
\begin{equation}\label{eq:ideal_norm}
N((I,\mu)) = \begin{vmatrix}a_1 & a_2 \\ b_1 & b_2 \end{vmatrix}
\end{equation}
where $\langle a_1 + a_2 \tau,b_1+b_2 \tau \rangle$ is a $\Z$-basis of $(I,\mu)$ which is oriented as per the orientation of $(I,\mu)$.
Then, the \emph{narrow class group} of $S$, denoted by $\text{Cl}^+(S)$, is defined as the group of equivalence classes of invertible oriented ideals of $S$. 

The next theorem provides a correspondence between binary quadratic forms and ideals in quadatic rings.

\begin{theorem}\label{thm:bqfidealcorrespondence}
Let $D$ be an integer congruent to $0$ or $1$ modulo $4$. There is a canonical bijection between the set of $\mathrm{SL}_2(\Z)$-classes of integral binary quadratic forms of discriminant $D$ and the set of oriented ideal classes of the oriented quadratic ring $S(D)$. Under this correspondence, primitive binary quadratic forms correspond to invertible oriented ideals and the correspondence induces a bijection between the set of $\text{SL}_2(\Z)$-classes of primitive binary quadratic forms of discriminant $D$ and the set of invertible oriented ideal classes of the oriented quadratic ring $S(D)$.
\end{theorem}

We recall that the set of $\text{SL}_2(\Z)$-equivalence classes of primitive binary quadratic forms of a fixed discriminant $D$ has the structure of a finite abelian group under composition, which is called the class group and denoted by $ \text{Cl}((\mathrm{Sym}^2 \Z^2)^*;D)$. The following well-known result describes the relation between the groups $\text{Cl}^+(S(D))$ and $ \text{Cl}((\mathrm{Sym}^2 \Z^2)^*;D)$.

\begin{theorem}
    Let $D$ be a nonzero integer congruent to $0$ or $1$ (mod $4$). There is a canonical isomorphism of groups
    $$ \mathrm{Cl}((\mathrm{Sym}^2 \Z^2)^*;D) \cong \mathrm{Cl}^+(S(D)).$$
\end{theorem}


\section{Composition of \texorpdfstring{$2 \times 2 \times 2$}{2 x 2 x 2} cubes}
\label{sec:cubecomp}
 
As stated in Section \ref{sec:Intro}, Bhargava (\cite{bhargava:2004}) presents an approach to Gauss composition using $2 \times 2 \times 2$ cubes and also proves a composition law on the space of cubes. We briefly discuss this before describing our result on composition identities for  cubes.

Let $A\in \Z^2\otimes\Z^2\otimes\Z^2.$  We represent $A$ pictorially by the cube
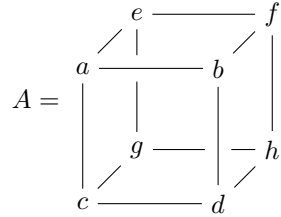
\begin{equation*}
A = \cube{a}{b}{c}{d}{e}{f}{g}{h}{}
\end{equation*}
with $a,b,c,d,e,f,g\in\Z.$
For convenience we will also sometimes use the notation
$$A=[a,b,c,d,e,f,g,h].
$$
The cube $A$
can be sliced into three pairs of $2 \times 2$ matrices $(M_i,N_i)$, for $i=1,2,3$, given by
\begin{equation*}
\begin{split}
        M_1 = \begin{bmatrix} a & b \\ c & d \end{bmatrix}, N_1 = \begin{bmatrix} e & f \\ g & h \end{bmatrix}, \\
        M_2 = \begin{bmatrix} a & e \\ b & f \end{bmatrix}, N_2 = \begin{bmatrix} c & g \\ d & h \end{bmatrix}, \\
        M_3 = \begin{bmatrix} a & c \\ e & g \end{bmatrix}, N_3 = \begin{bmatrix} b & d \\ f & h \end{bmatrix}.
\end{split}
\end{equation*}
Note that our $(M_2,N_2)$ and $(M_3,N_3)$ are swapped in comparison with \cite{bhargava:2004}. We define three binary quadratic forms associated to $A$ by 
\begin{equation}\label{eq:cubebqf}
Q_i^A(x,y) = -\text{det}(M_i x - N_i y).
\end{equation}
The three quadratic forms have the same discriminant, which we call also the discriminant of the cube $A$, denoted $\text{disc}(A)$. The cube $A$ is called \emph{projective} if the quadratic forms associated to $A$ are all primitive, and the cube is said to be \emph{nondegenerate} if its discriminant is nonzero. 

The group $\Gamma = \text{SL}_2(\Z) \times \text{SL}_2(\Z) \times \text{SL}_2(\Z)$ acts on a cube $A$ in the following way. The element $(\gamma, \text{id},\text{id})$ acts on $A$ by transforming the pair $(M_1,N_1)$ to $(pM_1+qN_1,rM_1+sN_1)$, where $\gamma = \begin{bmatrix} p & q \\ r & s \end{bmatrix}$. The actions of $(\text{id},\gamma,\text{id})$ and $(\text{id},\text{id},\gamma)$ are defined analogously, and the three actions commute with each other, thus defining an action of the group $\Gamma$ on the space of $2 \times 2 \times 2$ cubes. Also, if $i \neq j$, then the $j^{th}$-factor of $\text{SL}_2(\Z)$ in $\Gamma$ acts trivially on $Q_i^A$, and the $i^{th}$-factor transforms $Q_i^A(x,y)$ to $Q_i^A(px-ry,-qx+sy)$. This shows that the discriminant of the cube is preserved under the action of $\Gamma$. 

The Cube Law of Bhargava sets the formal sum of $Q_1^A,Q_2^A$ and $Q_3^A$ to be trivial in the free abelian group generated by the primitive binary quadratic forms of a fixed discriminant $D = \text{disc}(A)$. The primitive binary quadratic forms that are equivalent under the action of $\text{SL}_2(\Z)$ get identified by the relation imposed by the Cube Law, and after choosing the identity element to be the form $Q_{\text{id,D}}$ defined in  (\ref{eq:principal_form}), we obtain the relation of $\text{SL}_2(\Z)$ equivalence classes
\begin{align*}
    [Q_1^A] + [Q_2^A] + [Q_3^A] = [Q_{\text{id,D}}].
\end{align*}
Thus, Bhargava cubes give us Gauss composition. See equation (\ref{eq:lemmermeyerid}) for an explicit composition identity involving $Q_1^A,Q_2^A$ and $Q_3^A$.  In Section \ref{subsec:CubeLaw} we described how the Cube Law also gives rise to a group structure on the set of $\Gamma$-equivalence classes of projective cubes of discriminant $D$. The resulting group is denoted $\text{Cl}(\Z^2 \otimes \Z^2 \otimes \Z^2;D)$.

We will have need to introduce cubes constructed from simple (signed) permutations of the coefficients of $A$.  For the cube $A=[a,b,c,d,e,f,g,h]$ we define the associated cubes $ A^\iota, A^\sigma$ and $\tilde{A}$ by
\begin{align*}
  A^\iota &= [e,f,g,h,a,b,c,d] \\
A^\sigma &= [-e,-f,-g,-h,a,b,c,d]\\ 
\tilde{A}&=[-a,b,c,-d,e,-f,-g,h].
\end{align*}
As noted in the Introduction, if $A$ is projective, then the class of $\tilde A$ is the inverse of  $[A]$.
The next proposition shows that the binary quadratic forms associated to these new cubes are simply related to $Q_1^A,Q_2^A,Q_3^A.$

\begin{proposition} \label{prop:bqfs_Aiota,Asigma,Atilde}
  Let $A$ be a $2 \times 2 \times 2$ cube. Then, the quadratic forms associated to the cubes $A^\iota, A^\sigma$ and $\tilde{A}$ are given by
\begin{align*}
      Q_1^{A^\iota}(x,y) &= Q_1^A(y,x), \ \ \ Q_2^{A^\iota} = Q_2^A, Q_3^{A^\iota} = Q_3^A,\\
      Q_1^{A^\sigma}(x,y) &= Q_1^A(-y,x), \ Q_2^{A^\sigma} = Q_2^A, Q_3^{A^\sigma} = Q_3^A
\end{align*}
and
\begin{equation*}
    Q_i^{\tilde{A}}(x,y) = Q_i^A(x,-y)
\end{equation*}
for $i=1,2,3$.
\end{proposition}

\subsection{Correspondence between \texorpdfstring{$2 \times 2 \times 2$}{2 x 2 x 2} cubes and triples of ideals} \label{subsec:cubetriple}

We start by recalling the correspondence between $2 \times 2 \times 2$ cubes and \emph{balanced} triples of oriented ideals as given in \cite{bhargava:2004}. We follow the notation described in Section \ref{sec:Gausscomp}.

Let $D$ be an integer congruent to $\epsilon=0$ or 1 mod 4. A triple $(I_1,I_2,I_3)$ of oriented ideals of $S = S(D)$ is said to be \emph{balanced} if $I_1 I_2 I_3 \subset S$ and $N(I_1)N(I_2)N(I_3)=1.$
Two balanced triples $(I_1,I_2,I_3)$ and $(J_1,J_2,J_3)$ are said to be \textit{equivalent} if there exist $\lambda_1,\lambda_2,\lambda_3 \in K^{\times}$ such that $I_m=\lambda_m J_m$ for $m=1,2,3$. 

By Theorem 11 of \cite{bhargava:2004}, there is a bijection between the nondegenerate $\Gamma$-orbits of $\Z^2 \otimes \Z^2 \otimes \Z^2$ and the set of isomorphism classes of pairs $(S,(I_1,I_2,I_3))$, where $S$ is a nondegenerate oriented quadratic ring and $(I_1,I_2,I_3)$ is an equivalence class of balanced triple of oriented ideals of $S$. Under this bijection, the discriminant of a cube equals the discriminant of the corresponding oriented quadratic ring. The correspondence is obtained as follows. 

Suppose we are given a pair $(S,(I_1,I_2,I_3))$ as described above. Let $\langle 1, \tau \rangle$ be the positively oriented $\Z$-basis of $S$, where $\tau$ satisfies the equation (\ref{eq:taurelation}). Let $\langle \alpha_1, \alpha_2 \rangle, \langle \beta_1, \beta_2 \rangle$ and $\langle \gamma_1, \gamma_2 \rangle$ be $\Z$-bases for $I_1, I_2$ and $I_3$ respectively, which are oriented according to the respective orientation of the ideals. The balanced condition then implies that there exist 16 integers $a_{ijk}$ and $a'_{ijk}$ for $i,j,k \in \lbrace 1,2 \rbrace$ such that
\begin{equation} \label{eq:system_cubetriple} 
\alpha_i \beta_j \gamma_k = a'_{ijk} + a_{ijk} \tau.
\end{equation}
The cube corresponding to $(S,(I_1,I_2,I_3))$ is $A = \lbrace a_{ijk} \rbrace$. To go from a cube $A = \lbrace a_{ijk} \rbrace$ to a pair $(S,(I_1,I_2,I_3))$, we first note that the ring $S$ is determined by its discriminant which equals the discriminant of the cube. Then, we solve the system of equations given by (\ref{eq:system_cubetriple}) for $\alpha_i$'s, $\beta_j$'s and $\gamma_k$'s. Here, the values of $a'_{ijk}$'s are uniquely determined in terms of the coefficients $a_{ijk}$ as given in (\cite{bhargava:2004}). The corresponding triple of oriented ideals is determined uniquely up to equivalence. Indeed, we can choose the triple of ideals with $\Z$-bases given by the following.
\begin{equation} \label{eq:basesofideals}
    \begin{split}
        \alpha_1 &= a'_{111} + a_{111} \tau, \alpha_2 = a'_{211} + a_{211} \tau, \\
        \beta_1 &= a'_{212} + a_{212} \tau, \beta_2 = a'_{222} + a_{222} \tau, \\
        \gamma_1 &= 1/\beta_1, \gamma_2 = 1/\alpha_2.
    \end{split}
\end{equation}
We call a choice of ordered $\Z$-bases for $I_1, I_2, I_3$ \emph{compatible} with the cube $A$ if the equations (\ref{eq:system_cubetriple}) hold.  Henceforth whenever we introduce a cube with associated pair $(S, (I_1,I_2, I_3))$ we always assume the bases chosen for the $I_j$ are compatible in this sense.

The $S$-module structure of the ideals can be described in terms of the cube:

\begin{proposition}[\cite{bhargava:2004}, Eq. (20)]\label{prop:bhargava20}
  The $S$-module structure of $I_1$ is given by
\begin{equation}\label{eq:SmoduleI1}
    \begin{split}
        \tau \alpha_1 &= \frac{q_1+\epsilon}{2} \alpha_1 + p_1 \alpha_2, \\
        \tau \alpha_2 &= -r_1 \alpha_1 + \frac{\epsilon-q_1}{2} \alpha_2,
    \end{split}
\end{equation}
where $Q_1^A(x,y) = p_1x^2 + q_1 xy + r_1 y^2$, and $\epsilon \in \lbrace 0,1 \rbrace$ is congruent to $D$ (mod $4$).  The $S$-module structures of $I_2$ and $I_3$ are similarly given in terms of $Q_2^A$ and $Q_3^A$.
\end{proposition}

Under the correspondence between cubes and balanced triples of ideals, the projective orbits of $2 \times 2 \times 2$ cubes of discriminant $D$ correspond to balanced triples of projective ideals of $S(D)$. Also, the set of equivalence classes of balanced triples of projective oriented ideals has a natural structure of an abelian group under component-wise multiplication i.e. the product of the balanced triples $(I_1,I_2,I_3)$ and $(J_1,J_2,J_3)$ is $(I_1J_1,I_2J_2,I_3J_3)$. Thus, the set of projective orbits of $2 \times 2 \times 2$ cubes of discriminant $D$ gets an abelian group structure, which is denoted as $\text{Cl}(\Z^2 \otimes \Z^2 \otimes \Z^2; D)$. We also have the following group isomorphism:
$$ \mathrm{Cl}(\Z^2 \otimes \Z^2 \otimes \Z^2;D) \cong \mathrm{Cl}^+(S(D)) \times \mathrm{Cl}^+(S(D)).$$
This is obtained by sending a projective balanced triple $(I_1,I_2,I_3)$ to $(I_1,I_2)$, because the third ideal is determined from the first two. 

In the next three propositions we continue to let $A=(a_{ijk})$ be a cube corresponding to the balanced triple of oriented ideals $(I_1, I_2, I_3).$  We fix compatible bases $\langle \alpha_1, \alpha_2 \rangle, \langle \beta_1, \beta_2 \rangle$ and $\langle \gamma_1, \gamma_2 \rangle$ for the three ideals, respectively.

Recall that the companion cube $A'$ of a cube $A$ was defined in the introduction to be
\begin{equation}
  \label{def:companion_cube}
A'(\x,\y,\z)=A(\x'-\epsilon \x,\y,\z)\mbox{ for } \x,\y,\z\in\Z^2,
\end{equation}
for $$\x'=
\begin{pmatrix}
  \tfrac {q+\epsilon}2& -r\\ p & \tfrac{-q+\epsilon}2
\end{pmatrix}\x.
$$
The next proposition identifies the coefficients of the companion cube with the $a^{'}_{ijk}$ defined in (\ref{eq:system_cubetriple}).  In \cite{bhargava:2004} by comparison, Bhargava shows that the coefficients $a'_{ijk}$ (which he denotes by $c_{ijk}$ in his Eq. (15)) are determined by the cube $A$ by giving the unique explicit solution to the system of 18 linear and quadratic equations arising from the associativity and commutativity of multiplication in the ring $S$.

\begin{proposition} \label{prop:companioncube}
Let $A$ be a $2 \times 2 \times 2$ cube.
Then, the companion cube $A'$ has coefficients equal to $a'_{ijk}$.
\end{proposition}

\begin{proof}
By (\ref{eq:system_cubetriple}), for $\x = (x_1,x_2)^t, \y = (y_1,y_2)^t, \z = (z_1,z_2)^t \in \Z^2$, the cube $A$ defines a trilinear form $C$ satisfying
\begin{equation}\label{eq:trilinear1}
    (x_1 \alpha_1 + x_2 \alpha_2) (y_1 \beta_1 + y_2 \beta_2) (z_1 \gamma_1 + z_2 \gamma_2) = A(\x,\y,\z) \tau + C(\x,\y,\z).
\end{equation}
Let $a'_{ijk}$ be the coefficients of the cube $C.$  Our goal is to show that $A'=C.$

Multiplying both sides of equation (\ref{eq:trilinear1}) by $\tau$,
\begin{equation*}
    \tau (x_1 \alpha_1 + x_2 \alpha_2) (y_1 \beta_1 + y_2 \beta_2) (z_1 \gamma_1 + z_2 \gamma_2) = A(\x,\y,\z)\tau^2  + C(\x,\y,\z) \tau.
\end{equation*} 
Set
\begin{equation} \label{eq:xprime}
    \x'=\begin{pmatrix}
        x_1' \\ x_2'
    \end{pmatrix} = \begin{pmatrix}
     \frac{q_1 + \epsilon}{2} & -r_1 \\ p_1 &  \frac{-q_1 + \epsilon}{2}
    \end{pmatrix} \x,
\end{equation}
so that  $\tau (x_1 \alpha_1 + x_2 \alpha_2) = (x'_1 \alpha_1 + x'_2 \alpha_2)$ by Proposition \ref{prop:bhargava20}.
Consequently
\begin{equation*}
  A(\x,\y,\z) \tau^2  + C(\x,\y,\z) \tau=A(\x',\y,\z) \tau + C(\x',\y,\z).
\end{equation*}
Replace $\tau^2$ by $\epsilon\tau+\left(\frac{D - \epsilon}{4} \right)$ and compare coefficients of $\tau$ on both sides.  This yields
\begin{equation*}
\begin{split}
    C(\x,\y,\z) & = A(\x',\y,\z) - \epsilon A(\x,\y,\z)\\
        & = A(\x' - \epsilon \x,\y,\z) \\
        & = A'(\x,\y,\z).
\end{split}
\end{equation*}
This proves that the coefficients of the companion cube $A'$ are equal to $a_{ijk}'$.
\end{proof}

Define the matrix
\begin{equation*}
  \label{eq:LA}
L_i^A=
\begin{pmatrix}
\frac{q_i - \epsilon}{2} & -r_i \\ p_i &  \frac{-q_i - \epsilon}{2}
\end{pmatrix},
\end{equation*}
where $Q_i^A(x,y)=p_ix^2+q_ixy+r_iy^2$.  Arguing as in the previous proof, we see that
\begin{equation}
  \label{eq:companionL}
A'(\x,\y,\z)=A(L_1^A\x, \y, \z)=A(\x, L_2^A\y, \z), A(\x,\y,L_3^A\z).
\end{equation}

The next proposition gives the relation between the quadratic forms associated to a cube $A$ and a triple of ideals corresponding to it.

\begin{proposition}\label{prop:normQ}
With the choice of the $\Z$-bases of the ideals $I_1,I_2,I_3$ as given in equation (\ref{eq:basesofideals}), we have
\begin{equation} \label{eq:bqfnormform}
\begin{split}
    N(x \alpha_1 - y \alpha_2)/N(I_1) = Q_1^A(x,y), \\
    N(x \beta_1 - y \beta_2)/N(I_2) = Q_2^A(x,y), \\
    N(x \gamma_1 - y \gamma_2)/N(I_3) = Q_3^A(x,y). 
\end{split}
\end{equation}
\end{proposition}

\begin{proof}
This is a straightforward computation using Proposition \ref{prop:bhargava20} and  the formula for the $a'_{ijk}$'s given on page 235 of \cite{bhargava:2004}.
\end{proof}

Recall that $\alpha_i\beta_j\gamma_k= a_{ijk} \tau + a'_{ijk}.$  Our final result of this subsection gives a formula for the products $\beta_j\gamma_k.$

\begin{proposition} \label{lemma:cubebasesmultiplication}
For $j,k \in \{1,2\}$,
    $$ N(I_1) \beta_j \gamma_k = (a_{2jk}) \overline{\alpha_1} + (a_{1jk}) (-\overline{\alpha_2}). $$
Equivalently, we have
$$N(I_1)(y_1\beta_1+y_2\beta_2)(z_1\gamma_1+z_2\gamma_2)=
A(\underline{e}_2, \y,\z)\overline{\alpha_1} + A(\underline{e}_1, \y,\z) (-\overline{\alpha_2})
$$
for $\y=(y_1,y_2)^t, \z=(z_1,z_2)^t\in \Z^2.$
\end{proposition}

\begin{proof}
Multiplying Eq. \eqref{eq:system_cubetriple} by $\overline{\alpha_1}$ on both sides gives 
$$ N(\alpha_1) \beta_j \gamma_k = a'_{1jk} \overline{\alpha_1} + a_{1jk} \tau \overline{\alpha_1}.$$
Take conjugates on both sides of Eq. \eqref{eq:SmoduleI1} and simplify to get an expression for $\tau \overline{\alpha_1}$. Substituting that expression in the equation above, we get 
\begin{equation*}
\begin{split}
N(\alpha_1) \beta_j \gamma_k & = a'_{1jk} \overline{\alpha_1} + a_{1jk} \left( - \left( \frac{q_1 - \epsilon}{2} \right) \overline{\alpha_1} - p_1 \overline{\alpha_2} \right) \\
& = \left( a'_{1jk} - a_{1jk} \left( \frac{q_1 - \epsilon}{2} \right) \right) \overline{\alpha_1} + (a_{1jk} p_1) (-\overline{\alpha_2}).
\end{split}
\end{equation*}
Also, $\frac{N(\alpha_1)}{N(I_1)} = p_1$ by Prop. \ref{prop:normQ}. So we obtain
\begin{equation*}
p_1 N(I_1) \beta_j \gamma_k = \left( a'_{1jk} - a_{1jk} \left( \frac{q_1 - \epsilon}{2} \right) \right) \overline{\alpha_1} + (a_{1jk} p_1) (-\overline{\alpha_2}).
\end{equation*}
Similarly, multiplying Eq. \eqref{eq:system_cubetriple} by $\overline{\alpha_2}$ on both sides and arguing as above, we get
\begin{equation*}
r_1 N(I_1) \beta_j \gamma_k = (a_{2jk}r_1) \overline{\alpha_1} + \left( -a'_{2jk} - a_{2jk} \left( \frac{q_1 + \epsilon}{2} \right) \right)(- \overline{\alpha_2}).
\end{equation*}
Finally, comparing the last two equations yields
$$ N(I_1) \beta_j \gamma_k = (a_{2jk}) \overline{\alpha_1} + (a_{1jk}) (-\overline{\alpha_2}). $$

\end{proof}

\subsection{Composition identities for \texorpdfstring{$2 \times 2 \times 2$}{2 x 2 x 2} cubes} \label{subsec:compidcubes}

We restate our main result on composition identities for $2 \times 2 \times 2$ cubes from the introduction.

\cubecomp

\subsection{Proof of Theorem \ref{thm:cubecomp}, forward direction}
\label{sec:proof_cubes_forward}
\begin{proof}

Assume that $$ [A] + [B] + [C] = [\id] .$$ We choose the balanced triples $(I_1,I_2,I_3),(J_1,J_2,J_3),(K_1,K_2,K_3)$ corresponding to $A,B$ and $C$ so that $I_mJ_mK_m= S$ for $m=1,2,3$. Indeed, since the triples $(I_1J_1K_1,I_2J_2K_2,I_3J_3K_3)$ and $(S,S,S)$ are equivalent, we can rescale to assume that $I_mJ_mK_m= S$ for each $m.$

We choose compatible bases
\begin{itemize}
  \item $\langle \alpha_1, \alpha_2 \rangle, \langle \beta_1, \beta_2 \rangle,\langle \gamma_1, \gamma_2 \rangle$ for $I_1,I_2,I_3$
  \item $\langle \omega_1,\omega_2 \rangle$, $\langle \sigma_1,\sigma_2 \rangle$, $\langle \theta_1,\theta_2 \rangle$ for $J_1, J_2, J_3$
  \item $\langle \pi_1, \pi_2 \rangle, \langle \chi_1, \chi_2 \rangle, \langle \psi_1, \psi_2 \rangle$ for $K_1,K_2, K_3.$
\end{itemize}
Let $R = (r_{ijk}),S = (s_{ijk})$ and $T = (t_{ijk})$ be cubes corresponding to the balanced triples $(I_1,J_1,K_1), (I_2,J_2,K_2)$ and $(I_3,J_3,K_3)$, with the same choices for bases of the ideals. It follows from Proposition \ref{prop:normQ} that the cubes $(R,S,T)$ are dual to the cubes $(A,B,C)$.
We note that the cube $\tilde A$
defined in \eqref{def:Atilde} corresponds to the oriented triple  $(I_1^{-1},I_2^{-1},I_3^{-1})$ with compatible bases
$$\langle \alpha_1^\prime, -\alpha_2^\prime \rangle, \langle \beta_1^\prime, -\beta_2^\prime \rangle,\langle \gamma_1^\prime, -\gamma_2^\prime \rangle
$$
where $\alpha_i^\prime=N(I_1)^{-1}\widebar\alpha_i, \beta_i^\prime=N(I_2)^{-1}\widebar\beta_i,\gamma_i^\prime=N(I_3)^{-1}\widebar\gamma_i$ for $i=1,2.$

Applying Proposition \ref{lemma:cubebasesmultiplication} to the cubes $R,S$ and $T$, we get that
\begin{equation}\label{eq:idealproduct}
\begin{split}
(x_1 \omega_1 + x_2 \omega_2) (u_1 \pi_1 + u_2 \pi_2) = a_1 \alpha_1^\prime + a_2 \alpha_2^\prime, \\
(y_1 \sigma_1 + y_2 \sigma_2) (v_1 \chi_1 + v_2 \chi_2) = b_1 \beta_1^\prime + b_2 \beta_2^\prime, \\
(z_1 \theta_1 + z_2 \theta_2) (w_1 \psi_1 + w_2 \psi_2) = c_1 \gamma_1^\prime + c_2 \gamma_2^\prime,
\end{split}
\end{equation}
where $\aaa=(a_1,a_2)^t = R^\iota(\x,\uu),\bb=(b_1,b_2)^t = S^\iota(\y,\vv), \cc=(c_1,c_2)^t = T^\iota(\z,\ww)$.

Multiplying together the equations in \eqref{eq:idealproduct}, yields the identity
\begin{equation}
\label{eq:compid1}
    [B(\x,\y,\z) \tau + B'(\x,\y,\z)][C(\uu,\vv,\ww) \tau +C'(\uu,\vv,\ww)] =  \tilde{A}(\aaa,\bb,\cc) \tau + \tilde{A}'(\aaa,\bb,\cc) .
\end{equation}
The cube $\tilde{A}$ is a representative of the inverse class of the cube $A$ and  has the property that
\begin{equation} \label{eq:cube_inverse_relation}
\tilde{A}((a_1,a_2)^t,(b_1,b_2)^t,(c_1,c_2)^t) = A((-a_1,a_2)^t,(-b_1,b_2)^t,(-c_1,c_2)^t).
\end{equation}
Substituting this relation in \eqref{eq:compid1}, expanding and equating the coefficients of $\tau$ on both sides gives us
$$ (B \star C) (\x,\y,\z;\uu,\vv,\ww) = A(R^\sigma(\x,\uu),S^\sigma(\y,\vv), T^\sigma(\z,\ww)).$$

Lastly, we obtain
\begin{align*}
  Q_1^B(1,0)\cdot Q_1^C(1,0)&=Q_1^A(r_{211},r_{111}),\\
  Q_2^B(1,0)\cdot Q_2^C(1,0)&=Q_2^A(s_{211},s_{111}),\\
  Q_3^B(1,0)\cdot Q_3^C(1,0)&=Q_3^A(t_{211},t_{111}).
\end{align*}
by applying the equation (\ref{eq:lemmermeyerid}) to the cubes $R,S,T$ respectively.

\subsection{Proof of Theorem \ref{thm:cubecomp}, reverse direction.}
Now suppose $A, B, C$ are three projective cubes of common discriminant $D\equiv\epsilon$ mod 4. Let $\x,\y,\z,\uu,\vv,\ww\in \Z^2.$ Suppose there exist cubes $R=(r_{ijk}),S=(s_{ijk}),T=(t_{ijk})$ of discriminant $D$ satisfying the identity
\begin{align*}
(B \star C)(\x,\y,\z;\uu,\vv,\ww)
= A(R^\sigma(\x,\uu),S^\sigma(\y,\vv),T^\sigma(\z,\ww)),
\end{align*}
for all $\x,\y,\z,\uu,\vv,\ww\in \Z^2$.
We also assume (\ref{eq:cubecompQ}) and (\ref{eq:cubecompQ2}).
Using equation \eqref{eq:cube_inverse_relation}, we have that 
\begin{align*}
(B \star C)&(\x,\y,\z;\uu,\vv,\ww) 
= \tilde{A}(R^\iota(\x,\uu),S^\iota(\y,\vv),T^\iota(\z,\ww)).
\end{align*}
We will first show that the above identity implies that 
\begin{align}
\label{eq:compid}
    [B(\x,\y,\z) \tau &+ B'(\x,\y,\z)][C(\uu,\vv,\ww) \tau +C'(\uu,\vv,\ww)]
\\ \nn
&=  \tilde{A}(R^\iota(\x,\uu),S^\iota(\y,\vv),T^\iota(\z,\ww) \tau + \tilde{A}'(R^\iota(\x,\uu),S^\iota(\y,\vv),T^\iota(\z,\ww)).
\end{align}
Expanding out equation (\ref{eq:compid}), we see that it suffices to show 
\begin{align}
\label{eq:companion_compid}
    B'(\x,\y,\z)  C'(\uu,\vv,\ww) + \frac{(D - \epsilon)}{4}&B(\x,\y,\z) C(\uu,\vv,\ww) \\ \nn  &= \tilde{A}'(R^\iota(\x,\uu),S^\iota(\y,\vv),T^\iota(\z,\ww)).
\end{align}

Before proceeding, we explicate the group action on cubes when the cube is thought of as a pair of bilinear forms.
For $\gamma$ a two-by-two matrix, the action takes the form
\begin{align} \label{eq:action_bilinear}
  ((\gamma,\id,\id)\circ R)(\x,\uu)
&=\gamma
   ( R(\underline e_1, \x,\uu),R(\underline e_2, \x,\uu))^t\\ \nn
((\id,\gamma, \id)\circ R)(\x,\uu)
&= R(\gamma^t \x, \uu) \\ \nn
((\id,\id,\gamma)\circ R)(\x,\uu)&= R(\x,\gamma^t \uu).
\end{align}

By the definition of companion cube given in \eqref{def:companion_cube}, we have
\begin{align*}
    \tilde{A}'(R^\iota(\x,\uu),S^\iota(\y,\vv),T^\iota(\z,\ww)) = \tilde{A}(L_1^{\tilde A}R^\iota(\x,\uu),S^\iota(\y,\vv),T^\iota(\z,\ww)).
\end{align*}
The first argument of $\tilde A$ is
\begin{equation}
  \label{eq:L1tildeA}
  L_1^{\tilde A}R^\iota(\x,\uu)=L_1^{\tilde A}
  \begin{pmatrix}
    0&1\\1&0
  \end{pmatrix}
  R(\x,\uu)=
  \begin{pmatrix}
    0&1\\1&0
  \end{pmatrix}(L_1^{A})^tR(\x,\uu).
\end{equation}
Using the description of the group action from \eqref{eq:action_bilinear}, the last expression becomes
 \begin{align}\nn
 \left( \left( \begin{pmatrix}
     0&1\\1&0
   \end{pmatrix}  (L_1^{A})^t, \id, \id\right) \circ  R \right) & (\x,\uu)
 =
 \left( \left( \begin{pmatrix}
     0&1\\1&0
   \end{pmatrix}(L_1^{R})^t, \id, \id\right) \circ R \right)(\x,\uu),
 \mbox{  by (\ref{eq:cubecompQ})}\\ \nn
 &= \left( \left( \begin{pmatrix}
     0&1\\1&0
   \end{pmatrix}, (L_2^{R})^t, \id\right) \circ R \right)(\x,\uu),
 \mbox{  by (\ref{eq:companionL})}\\ \nn
 &=\left( \left( \begin{pmatrix}
     0&1\\1&0
   \end{pmatrix}, \id, \id\right)\circ R \right) (L_2^R\x,\uu)\\ \nn
 &=R^\iota(L_2^R\x,\uu).
 \end{align}
 Hence,
 \begin{align*}
   \tilde{A}'&(R^\iota(\x,\uu),S^\iota(\y,\vv),T^\iota(\z,\ww)) \\ \nn
 &=\tilde{A}(R^\iota(L_2^R\x,\uu),S^\iota(\y,\vv),T^\iota(\z,\ww))\\ \nn
 &=(B\star C)(L_2^R\x,\y,\z;\uu,\vv,\ww)\\ \nn
 &= B(L_2^R\x,\y,\z)C^\prime(\uu,\vv,\ww)+B^\prime(L_2^R\x,\y,\z)C(\uu,\vv,\ww)+\epsilon B(L_2^R\x,\y,\z)C(\uu,\vv,\ww)\\ \nn
&=B^\prime(\x,\y,\z)C^\prime(\uu,\vv,\ww)+B^\prime(L_1^B\x,\y,\z)C(\uu,\vv,\ww)+\epsilon B^\prime(\x,\y,\z)C(\uu,\vv,\ww),
\end{align*}
since $L_2^R=L_1^B$.  Lastly, it is easy to check that $B'(L_1^B \x,\y,\z) = \left( \frac{D - \epsilon}{4} \right) B(\x,\y,\z)$, thus establishing equations (\ref{eq:companion_compid}) and (\ref{eq:compid}).

Now, we proceed to show how equation (\ref{eq:compid}) implies that $[A] + [B] + [C] = [\text{id}]$, or equivalently, $[B] + [C] = [\tilde{A}]$. We continue to use the notation from the start of Section \ref{sec:proof_cubes_forward} for the balanced triples and compatible bases corresponding to $A,B,C$ and $\tilde A.$

Using equation (\ref{eq:compid}), we get that
\begin{multline*}
    (x_1 \omega_1 + x_2 \omega_2) (y_1 \sigma_1 + y_2 \sigma_2) (z_1 \theta_1 + z_2 \theta_2) (u_1 \pi_1 + u_2 \pi_2) \\ (v_1 \chi_1 + v_2 \chi_2) (w_1 \psi_1 + w_2 \psi_2) = (a_1 \alpha_1^\prime + a_2 \alpha_2^\prime)(b_1 \beta_1^\prime + b_2 \beta_2^\prime)(c_1 \gamma_1^\prime + c_2 \gamma_2^\prime),
\end{multline*}
where $(a_1,a_2)^t = R^\iota(\x,\uu), (b_1,b_2)^t = S^\iota(\y,\vv), (c_1,c_2)^t = T^\iota(\z,\ww).$ We substitute $y_1=z_1=1, y_2=z_2=0, v_1=w_1=1$ and $v_2=w_2=0$ in the above equation to get
\begin{multline*}
        (x_1 \omega_1 + x_2 \omega_2) \sigma_1 \theta_1 (u_1 \pi_1 + u_2 \pi_2) \chi_1 \psi_1 \\ = (a_1 \alpha_1^\prime + a_2 \alpha_2^\prime) ( s_{211} \beta_1^\prime + s_{111} \beta_2^\prime) ( t_{211} \gamma_1^\prime + t_{111} \gamma_2^\prime).
\end{multline*}
This shows that $J_1 K_1 \subset\alpha I_1^{-1}$ for some $\alpha \in K$. If $N(J_1 K_1) = N(\alpha) N(I_1^{-1})$, then $J_1 K_1$ and $I_1^{-1}$ are equivalent (oriented) ideals. Therefore we need to show that
$$N(J_1) N(K_1) N(\sigma_1 \theta_1 \chi_1 \psi_1) = N(I_1^{-1}) N(( s_{211} \beta_1^\prime + s_{111} \beta_2^\prime) ( t_{211} \gamma_1^\prime + t_{111} \gamma_2^\prime)).$$
By Proposition \ref{prop:normQ}, the previous equality is equivalent to 
\begin{multline*} \label{eq:normtriple}
N(J_1) N(K_1) Q_2^B(1,0) N(J_2) Q_3^B(1,0) N(J_3) Q_2^C(1,0) N(K_2) Q_3^C(1,0) N(K_3) \\ = N(I_1^{-1}) N(I_2^{-1}) Q_2^{\tilde{A}}(s_{211},-s_{111}) N(I_3^{-1}) Q_3^{\tilde{A}}(t_{211},-t_{111}).
\end{multline*}
The above equation follows from the second and third equalities in \eqref{eq:cubecompQ2}, and hence $J_1 K_1 = I_1^{-1}$. Similarly, we deduce that $J_2 K_2 = I_2^{-1}$ and $J_3 K_3 = I_3^{-1}$. We conclude that $[B] + [C] = [\tilde{A}]$, thus completing the proof of the theorem.

\end{proof}

\subsection{Example}\label{subsec:cubesexample}

We illustrate the composition identity of $2 \times 2 \times 2$ cubes with an example.

Consider the cubes $A,B$ and $C$ shown below. 
\begin{equation*} 
A = \cube{0}{-1}{-2}{-1}{-1}{0}{0}{6}, 
B = \cube{0}{1}{2}{0}{1}{0}{-1}{-6},
C = \cube{0}{1}{4}{-1}{1}{0}{0}{-3}.
\end{equation*}
The cubes are of discriminant $D = -47$ and they satisfy 
$$[A] + [B] + [C] = [\text{id}]$$
in $\text{Cl}(\Z^2 \otimes \Z^2 \otimes \Z^2;-47).$ Also, the associated quadratic forms are given by $Q_1^A = Q_1^B = [2,1,6], Q_1^C = [4,-1,3], Q_2^A = [1,-1,12], Q_2^B = [1,1,12], Q_2^C = [1,1,12], Q_3^A = Q_3^B = [2,-1,6]$ and $Q_3^C = [4,1,3].$

The cubes $(R,S,T)$ given by
\begin{equation*} 
R = \cube{0}{-1}{-2}{0}{-2}{0}{1}{3}, 
S = \cube{0}{1}{1}{-1}{1}{0}{0}{-12},
T = \cube{0}{1}{2}{0}{2}{0}{1}{-3}.
\end{equation*}
are dual to the cubes $(A,B,C)$. The cubes $A,B,C$ satisfy the composition identity 
\begin{align*} 
(B\star C)&(\x,\y,\z;\uu,\vv,\ww)\\ \nn
&= A(R^\sigma(\x,\uu), S^\sigma(\y,\vv), T^\sigma(\z,\ww))
\end{align*}
and
\begin{align*}
  Q_1^B(1,0)\cdot Q_1^C(1,0)&=Q_1^A(r_{211}, r_{111})\\
  Q_2^B(1,0)\cdot Q_2^C(1,0)&=Q_2^A(s_{211}, s_{111})\\
  Q_3^B(1,0)\cdot Q_3^C(1,0)&=Q_3^A(t_{211}, t_{111}).
\end{align*}

\section{Composition of binary cubic forms} \label{sec:bcfcomp}  

We next consider the space $\mathrm{Sym}^3 \Z^2$ of integral binary cubic forms with triplicate central coefficients, that is, polynomials $f(x,y)$ of the form $a_0x^3 + 3a_1x^2y + 3a_2xy^2 + a_3y^3$. Such a binary cubic form has a natural representation as the triply symmetric cube
\begin{equation}
  \label{eq:triplysymmetriccube}
A_f = \cube{a_0}{a_1}{a_1}{a_2}{a_1}{a_2}{a_2}{a_3}.
\end{equation}
Indeed the polarization formula gives a bijective correspondence between binary cubic forms on $\Z^2$ and symmetric trilinear forms on $\Z^2\times\Z^2\times\Z^2.$

The main result of this section, Theorem \ref{thm:bcf1}, is an explicit composition formula for binary cubic forms.  Such a result has previously been established by B. Nativi (\cite{nativi2019analogue}). Under the hood, our methods are essentially the same as his, however, using our results from the previous section, we are able to give a more streamlined proof. See Remark \ref{rmk:nativi} for a further comparison with Nativi's work.

We briefly mention some other related works on composition laws for binary cubic forms.  The work of Slupinski-Stanton (\cite{Slupinski2012}) provides a detailed analysis of the structure of binary cubic forms (over any field with characteristic not equal to 2 or 3) from the perspective of symplectic geometry, and also describes the group structure on the classes of binary cubic forms. (See Theorems 3.34 and 3.46 in \cite{Slupinski2012}.) This is investigated more closely in the PhD thesis of G. Yassiyevich (\cite{yassiyevich2020arithmetic}) which describes the relation between the class groups in the integral and rational cases. 

Classically, binary cubic forms have also been investigated from the perspective of invariant theory, for example, by Cayley and Eisenstein. In fact, the syzygy we obtain in Section \ref{subsec:bcfsyzygy} was proved by Eisenstein (\cite{Eisenstein+1844+89+104}).  There is also an approach to Gauss composition of binary quadratic forms using Clifford algebras, due to M. Kneser (\cite{kneser1982composition}), which was used by Hoffman-Morales (\cite{hoffman2000arithmetic}) to generalize Eisenstein's theory of binary cubic forms to binary cubic forms over arbitrary integral domains with characteristic not equal to 2 or 3.

\subsection{The group law for binary cubic forms}
\label{subsec:bcfs_group_law}
We follow the notation and definitions from Section \ref{sec:Gausscomp}.

Let $f(x,y)=a_0x^3 + 3a_1x^2y + 3a_2xy^2 + a_3y^3$ be a binary cubic form associated to the triply-symmetric cube $A_f$ as in (\ref{eq:triplysymmetriccube}). If $A_f$ is projective then $f$ is also called projective. The discriminant of $f(x,y)$ is defined to be the discriminant of the corresponding cube $A_f$. Explicitly, it is given by
$$\text{disc}(f) = -3a_1^2a_2^2 + 4a_0a_2^3 + 4a_1^3a_3 - 6a_0a_1a_2a_3  + a_0^2a_3^2.$$
We say that $f(x,y)$ is nondegenerate if its discriminant is nonzero.

In \cite{bhargava:2004}, Bhargava constructs a canonical bijection between the set of nondegenerate $\text{SL}_2(\Z)$-orbits of $\text{Sym}^3\Z^2$ and the set of equivalence classes of triples $(S,I,\delta)$, where $S$ is a nondegenerate oriented quadratic ring, $I$ is an (unoriented) ideal of $S$, and $\delta$ is an invertible element of $S \otimes \mathbb{Q}$ such that $I^3 \subset \delta \cdot S$ and $N(I)^3 = N(\delta)$. Two triples $(S,I,\delta)$ and $(S',I',\delta')$ are said to be equivalent if there exists an isomorphism $\phi: S \rightarrow S'$ and an element $\kappa \in S' \otimes \Q$ such that $I' = \kappa \phi(I)$ and $\delta'=\kappa^3 \phi(\delta)$. Under this bijection, the discriminant of a binary cubic form $f$ equals the discriminant of the corresponding oriented quadratic ring, and the projective binary cubic forms correspond to triples $(S,I,\delta)$ such that $I$ is a projective $S$-module. Since, the classes of triples $(S,I,\delta)$ have a natural structure of an abelian group, the set of projective $\text{SL}_2(\Z)$-orbits of $\text{Sym}^3\Z^2$ of a fixed discriminant $D$ is also endowed with the structure of an abelian group, which is denoted by $\text{Cl}(\text{Sym}^3 \Z^2;D)$. The identity element of the group $[f_{\text{id},D}] (=[\text{id}])$ is given by the equivalence class of the form
\begin{equation*}
    f_{\text{id},D} = \begin{cases} 3x^2y + \frac D4 y^3 & \text{if } D \equiv 0 \text{ (mod } 4) \\
    3x^2y+3xy^2 + \frac{D+3}{4}y^3 & \text{if } D \equiv 1 \text{ (mod } 4) \end{cases}
\end{equation*}
corresponding to the triply symmetric cubes $A_{\text{id},D}$ given in (\ref{def:id_cube}).

We now explain how the bijection is constructed. Suppose we start with a triple $(S,I,\delta)$ as above. Let $\langle 1, \tau \rangle$ be a positively oriented $\Z$-basis of $S$, where $\tau$ satisfies \eqref{eq:taurelation}. Let $\langle \alpha, \beta \rangle$ be a positively oriented $\Z$-basis for $I$. Since $I^3 \subset \delta \cdot S$, there exist $a_i,a_i' \in \Z$, for $i=0,1,2,3$ such that 
\begin{equation}
\label{eq:bcfsyseq}
\begin{split}
    \alpha^3 &= \delta(a_0'+a_0 \tau) \\
    \alpha^2 \beta &= \delta(a_1'+a_1 \tau) \\
    \alpha \beta^2 &= \delta(a_2'+a_2 \tau) \\
    \beta^3 &= \delta(a_3'+a_3 \tau).
\end{split}
\end{equation}
The binary cubic form corresponding to the triple $(S,I,\delta)$ is given by $$f(x,y) = a_0x^3 + 3a_1x^2y + 3a_2xy^2 + a_3y^3.$$ The binary cubic form $f'(x,y)= a_0'x^3 + 3a_1'x^2y + 3a_2'xy^2 + a_3'y^3$ is called the companion form of $f(x,y)$, and it is easily verified that the cube $A_{f'}$ is the companion cube of $A_f$ defined in \eqref{def:companion_cube}.
In Section \ref{subsec:bcfsyzygy}, we will see that the companion binary cubic form $f'$ is a covariant of $f$ and there is a syzygy connecting them.

In the other direction, to get a triple $(S,I_f,\delta_f)$ corresponding to a given binary cubic form $f(x,y) \in \text{Sym}^3 \Z^2$, we solve the system of equations in (\ref{eq:bcfsyseq}). The numbers $a'_0,a'_1,a'_2,a'_3$ are determined in terms of $a_0,a_1,a_2,a_3$ as given on page 238 in \cite{bhargava:2004}. Then, the values of $\alpha$ and $\beta$ are determined uniquely up to multiplication by a scalar $\kappa \in K$. The scalar $\delta_f$ is uniquely determined by $\alpha$ and $\beta$, and it gets multiplied by $\kappa^3$ if $\alpha$ and $\beta$ are multiplied by $\kappa$. A possible choice of values for $\alpha,\beta$ and $\delta_f$ are
\begin{equation*} \label{eq:Ibasisdelta}
\begin{split}
\alpha & = a'_1 + a_1 \tau, \\
\beta & = a'_2 + a_2 \tau, \\
\delta_f & = \alpha \beta.
\end{split}
\end{equation*}
It follows from (\ref{eq:bcfsyseq}) that for $x,y \in \Z$, we can write
\begin{equation*}
\label{eq:bcfproduct0}
    (x \alpha + y \beta)^3 = \delta_f(f(x,y) \tau + f'(x,y)).
\end{equation*}

We can express the group law on $\text{Cl}(\text{Sym}^3 \Z^2;D)$ in terms of the group law on cubes  $\text{Cl}(\Z^2\otimes\Z^2\otimes\Z^2;D)$ as follows.  
\begin{proposition} \label{prop:bcf_classgroup_cubes_classgroup}
Let $f,g,h$ be three projective binary cubic forms of discriminant $D$ in $\mathrm{Sym}^3 \Z^2$ with corresponding triply-symmetric cubes $A_f, A_g$ and $A_h,$  respectively.  Then
$$[f]+[g]+[h]=[\id]$$
in $\mathrm{Cl}(\mathrm{Sym}^3 \Z^2;D)$ if and only if
$$[A_f]+[A_g]+[A_h]=[\id]
$$
in $\mathrm{Cl}(\Z^2\otimes\Z^2\otimes\Z^2;D)$ and $\delta_f\delta_g\delta_h$ is in $K^{\times3}$, where $K^{\times3}$ denotes the group of cubes of units of $K$. 
\end{proposition}

\begin{proof}
    Let $f,g$ and $h$ correspond to the triples $(S,I_1,\delta_f), (S,I_2,\delta_g)$ and $(S, I_3,\delta_h)$ respectively. The cubes $A_f,A_g$ and $A_h$ correspond to the balanced triples of ideals $(I_1,I_1, \delta_f^{-1}I_1),(I_2,I_2, \delta_g^{-1}I_2)$ and $(I_3,I_3, \delta_h^{-1}I_3)$ respectively. Then, $[f]+[g]+[h] = [\text{id}]$ if and only if $I_1I_2I_3=(\delta)$ and $\delta_f \delta_g\delta_h=\delta^3$ for some $\delta\in K^{\times}$, which is equivalent to $[A_f]+[A_g]+[A_h]=[\id]$.
\end{proof}

We note that the binary quadratic forms $Q_i^{A_f}$  associated to the triply symmetric cube $A_f$ satisfy $Q_1^{A_f}=Q_2^{A_f}=Q_3^{A_f}$ and we denote the common binary quadratic form by $Q_f$.

\subsection{Explicit composition for binary cubic forms}
We now define the form product of two projective binary cubic forms of common discriminant $D$.  For two such forms $f$ and $g$, define
\begin{align*}
    (f \star g)((x,y);(u,v)) = f(x,y)g'(u,v) + f'(x,y)g(u,v) + \epsilon f(x,y) g(u,v)
\end{align*}
where as usual, $\epsilon=0$ or $1$ is congruent to $D$ mod 4.  Our main theorem for binary cubic forms is the following.


\begin{theorem} \label{thm:bcf1}
  Let $f, g, h \in \text{Sym}^3 \Z^2$ be three projective binary cubic forms of common discriminant $D$ and $\epsilon\in\{0,1\}$ congruent to $D$ mod 4. Let $x,y,u,v\in \Z.$ Then
$$[f]+[g]+[h]=[\id]$$
 if and only if there exists a cube $R$ of discriminant $D$ satisfying
 \begin{align*}
 (g\star h)&((x,y);(u,v))\\ \nn
 &= f(R^\sigma((x,y),(u,v))).
 \end{align*}
 and
\begin{equation}\label{eq:bcfcompQ}
         Q_1^{R} = Q^f, Q_2^{R} = Q^g,
\end{equation}
and
\begin{align} \label{eq:bcfcompQ2}
  Q^g(1,0)\cdot Q^h(1,0)&=Q^f(r_{211},r_{111}).
\end{align}
\end{theorem}

\begin{proof}

Suppose that $[f]+[g]+[h]=[\id].$  Then $[A_f]+[A_g]+[A_h]=[\id]$ in $\text{Cl}(\Z^2\otimes\Z^2\otimes\Z^2;D).$  By Theorem \ref{thm:cubecomp}, there exist cubes $R,S,T$ of discriminant $D$ satisfying
\begin{align} \label{eq:bcf_cube_identity}
(A_g\star A_h)&(\x,\y,\z;\uu,\vv,\ww)\\ \nn
&= A_f(R^\sigma(\x,\uu), S^\sigma(\y,\vv), T^\sigma(\z,\ww))
\end{align}
and satisfying \eqref{eq:bcfcompQ} and \eqref{eq:bcfcompQ2}.  In fact, it follows from the proof on Theorem \ref{thm:cubecomp} that we can take $R=S=T.$  Indeed, let the three cubes $A_f, A_g, A_h$ correspond to the oriented triples of ideals $(I_1,I_1, \delta_f^{-1}I_1),(I_2,I_2, \delta_g^{-1}I_2)$ and $(I_3,I_3, \delta_h^{-1}I_3)$, respectively.  Then, $R$ and $S$ will both correspond to $(I_1,I_2,I_3)$ and $T$ to $(\delta_f^{-1} I_1, \delta_g^{-1}I_2, \delta_h^{-1}I_3).$  Assuming without loss of generality that $\delta_f\delta_g\delta_h=1$ we see that $R=S=T.$
Thus, taking $\x=\y=\z$ and $\uu=\vv=\ww$ in \eqref{eq:bcf_cube_identity}, we conclude
$$(g\star h)((x,y);(u,v))
 = f(R^\sigma((x,y),(u,v))).
$$
Equation \eqref{eq:bcfcompQ2} follows from applying equation \eqref{eq:lemmermeyerid} to the cube $R$.

Conversely, let us assume that
\begin{equation}
  \label{eq:gstarh=f}
(g\star h)((x,y);(u,v))
 = f(R^\sigma((x,y),(u,v)))
\end{equation}
and \eqref{eq:bcfcompQ} and \eqref{eq:bcfcompQ2} are satisfied. We will show that the two multilinear forms
\begin{equation}
  \label{eq:AgstarAh}
  (A_g\star A_h)(\x,\y,\z;\uu,\vv,\ww)
\end{equation}
and
\begin{eqnarray}
  \label{eq:AfRsigma}
  A_f(R^\sigma(\x,\uu), R^\sigma(\y,\vv), R^\sigma(\z,\ww))
\end{eqnarray}
in $\x,\y,\z,\uu,\vv,\ww\in \Z^2$ are equal.  Both are symmetric in $\x,\y,\z$ and in $\uu,\vv,\ww$ --- this is obvious in the first case, and a routine computation in the second case.  The polarization formula applied to (\ref{eq:gstarh=f}) implies that  (\ref{eq:AgstarAh}) and (\ref{eq:AfRsigma}) are equal.  See \cite{kraft-procesi} Lemma 4.6, for example.
We conclude that
$[A_f]+[A_g]+[A_h]=[\id]$ by again appealing to Theorem \ref{thm:cubecomp}.

As before let the three cubes $A_f, A_g, A_h$ correspond to the oriented triples of ideals $(I_1,I_1, \delta_f^{-1}I_1),(I_2,I_2, \delta_g^{-1}I_2)$ and $(I_3,I_3, \delta_h^{-1}I_3)$, respectively.  Choose $\Z$-bases $\langle \alpha_j, \beta_j \rangle$ of $I_j$, for $j=1,2,3.$ The cube $\tilde{A_f}$ corresponds to the triple $(I_1^{-1},I_1^{-1},\delta_f I_1^{-1})$, and a compatible choice of basis is given by  
$\langle \alpha_1',-\beta_1' \rangle$ where $\alpha_1' = N(I_1)^{-1} \overline{\alpha_1}$ and $\beta_1' = N(I_1)^{-1} \overline{\beta_1}.$
From \eqref{eq:gstarh=f}, we get
\begin{align*}
    \frac{1}{\delta_g}(x \alpha_2 + y \beta_2)^3 \frac{1}{\delta_h} (u \alpha_3 + v \beta_3)^3= {\delta_f}(a \alpha_1' + b\beta_1')^3,
\end{align*}
where $(a,b) = R^{\iota}((x,y),(u,v)).$ Substituting $x=u=1, y=v=0$ in the above equation, we get
\begin{align*}
    \frac{\alpha_2 ^3 \alpha_3^3}{\delta_g \delta_h} = \delta_f (r_{211} \alpha_3 + r_{111} \beta_3)^3.
\end{align*}
Hence, $\delta_f \delta_g \delta_h$ is a cube of an element of $K$.  This fact combined with $[A_f]+[A_g]+[A_h]=[\id]$ allows us to conclude that $[f]+[g]+[h] = [\text{id}]$, as desired.

\end{proof}

\begin{remark} \label{rmk:nativi}
    In \cite{nativi2019analogue}, the definition of composition of two binary cubic forms is essentially identical to what we defined as form product of the cubic forms. We note that the $\Z$-basis of $S(D)$ used in \cite{nativi2019analogue} is slightly different from our choice.

    We also point out that Theorem 1.1 in \cite{nativi2019analogue} is stated incorrectly. If $f,g$ and $h$ are projective binary cubic forms of common discriminant $D$ satisfying the composition identity \eqref{eq:gstarh=f}, then it is not true that $[f] + [g] + [h] = [\mathrm{id}]$ in $\mathrm{Cl}(\text{Sym}^3 \Z^2;D)$ as is claimed in \cite{nativi2019analogue}. This can be seen by replacing $f$ by $\tilde{f}$, which is the form in the inverse class of $f$, in the identity \eqref{eq:gstarh=f} and changing the bilinear form accordingly. (One needs to assume the conditions on the bilinear form given by the cube $R$ as stated in our Theorem \ref{thm:bcf1} in order to derive the conclusion $[f] + [g] + [h] = [\mathrm{id}]$.)

\end{remark}

\subsection{Companion cubic form and syzygy} \label{subsec:bcfsyzygy}

The binary cubic form $T_f(x,y)$ defined by 
$$ T_f(x,y) = f'(x,y) + \frac{\epsilon}{2} f(x,y),$$
is classically referred to as the ``cubicovariant" of the binary cubic form $f$. The binary cubic forms $f',T_f$ and the quadratic form $Q_f$ associated to the triply symmetric cube $A_f$ are all covariants of $f$. Furthermore,
they are connected via the following syzygy, which is due to Eisenstein (\cite{Eisenstein+1844+89+104}):
$$ T_f(x,y)^2 - \frac{D}{4} f(x,y)^2 = Q_f(x,-y)^3.$$
Here $D$ is the discriminant of $f(x,y)$ (and $Q_f(x,y)$).

\subsection{Example}
We demonstrate an example of the composition identity involving binary cubic forms.

Let $f(x,y) = 3x^2y+2y^3, g(x,y) = x^3+3x^2y+6xy^2+2y^3$ and $h(x,y) = x^3+3xy^2+2y^3.$ Then, $f,g$ and $h$ have discriminant $8$ and satisfy $[f]+[g]+[h]=[\text{id}]$ in $\text{Cl}(\text{Sym}^3 \Z^2;8)$. The companion cubic forms are $f'(x,y) = x^3+6xy^2, g'(x,y) = -x^3-6x^2y-6xy^2-4y^3$ and $h'(x,y) = x^3-3x^2y-3xy^2-3y^3$. The binary quadratic forms $Q^f,Q^g,Q^h$ are given by $Q^f=[1,0,-2],Q^g=[-1,0,2],Q^h=[-1,2,1]$.

The cube $R$ given by
$$R = \cube{0}{-1}{-1}{-1}{1}{1}{0}{2}{}$$
is dual to the triple $(f,g,h)$, and we have the composition identity
\begin{align*}
(g\star h)&((x,y);(u,v)) 
= f(R^\sigma((x,y),(u,v))).
\end{align*}

\section{Composition of pairs of binary quadratic forms} \label{sec:pbqfcomp}

Now, we consider the composition law on the space $\Z^2 \otimes \text{Sym}^2 \Z^2$ of pairs of binary quadratic forms with even middle coefficients. An element $F = (F_1,F_2) = (ax^2+2bxy+cy^2,dx^2+2exy+fy^2) \in \Z^2 \otimes \text{Sym}^2 \Z^2$ can be represented as a doubly symmetric cube given by 
$$A_F = \cube{a}{b}{b}{c}{d}{e}{e}{f}.$$

The group $\text{SL}_2(\Z) \times \text{SL}_2(\Z)$ acts on the space $\Z^2 \otimes \text{Sym}^2 \Z^2$ in the following way. If $(g,h) \in \text{SL}_2(\Z) \times \text{SL}_2(\Z)$ and $(F_1,F_2) \in \Z^2 \otimes \text{Sym}^2 \Z^2$, then 
\begin{equation} 
\label{eq:sl2xsl2}
(g,h) \cdot (F_1,F_2) = (p(h^t \cdot F_1) + q(h^t \cdot F_2), r(h^t \cdot F_1) + s(h^t \cdot F_2)) ,
\end{equation}
where $g = \begin{pmatrix} p & q \\ r & s \end{pmatrix}$, $h^t$ denotes transpose of $h$, and the action of $h^t$ on $F_1$ and $F_2$ is as given in equation \eqref{eq:sl2action}. Under the inclusion map $\Z^2 \otimes \text{Sym}^2 \Z^2 \hookrightarrow \Z^2 \otimes \Z^2 \otimes \Z^2$, the action given in equation \eqref{eq:sl2xsl2} corresponds to $(g,h,h) \in \Gamma$ acting on the doubly symmetric cube $A_F$ corresponding to $F=(F_1,F_2)$. 
We note that the binary quadratic forms $Q_i^{A_F}$  associated to the doubly symmetric cube $A_F$ satisfy $Q_2^{A_F}=Q_3^{A_F}$.
The discriminant of the cube is the unique polynomial invariant for this action.
The pair $F=(F_1,F_2)$ is called projective (resp. nondegenerate) if the cube $A_F$ is projective (resp. nondegenerate).

\subsection{Correspondence between pairs of binary quadratic forms and pairs of ideals}

With the same notation as Section \ref{sec:Gausscomp}, we briefly describe the correspondence between pairs of binary quadratic forms and certain balanced triples of oriented ideals.

There is a canonical bijection between nondegenerate $\text{SL}_2(\Z) \times \text{SL}_2(\Z)$-orbits of $\Z^2 \otimes \text{Sym}^2 \Z^2$ and isomorphism classes of pairs $(S,(I_1,I_2,I_3))$, where $S$ is a nondegenerate oriented quadratic ring and $(I_1,I_2,I_3)$ is an equivalence class of balanced triples of oriented ideals of $S$ such that $I_2=I_3$. Under this bijection, the discriminant of an element of $\Z^2 \otimes \text{Sym}^2 \Z^2$ equals the discriminant of the corresponding oriented quadratic ring. 
The proof of this bijection is similar to the case of $2 \times 2 \times 2$ cubes.

The projective $\text{SL}_2(\Z) \times \text{SL}_2(\Z)$-orbits of $\Z^2 \otimes \text{Sym}^2 \Z^2$ correspond to pairs $(S,(I_1,I_2,I_2))$, where $I_1$ and $I_2$ are projective $S$-modules. Hence, the set of projective $\text{SL}_2(\Z) \times \text{SL}_2(\Z)$-orbits of $\Z^2 \otimes \text{Sym}^2 \Z^2$ of a fixed discriminant $D$ is equpped with the structure of an abelian group, and this group is denoted by $\text{Cl}(\Z^2 \otimes \text{Sym}^2 \Z^2;D)$. The identity element of the group is the class of the pair $F_{\text{id},D}$ given by 
\begin{equation*}
    F_{\text{id},D} = \begin{cases}
        \left( 2xy,x^2+\frac D4 y^2 \right) & \text{if } D \equiv 0 \text{ (mod } 4) \\
        \left( 2xy+y^2,x^2+2xy+\frac{D+3}{4} y^2 \right) & \text{if } D \equiv 1 \text{ (mod } 4).
    \end{cases}
\end{equation*}
We see that the ideal class of $I_1$ in a balanced triple $(I_1,I_2,I_2)$ is determined by the ideal class of $I_2$. Hence, we get the following isomorphism:
\begin{equation*} \label{eq:bqfandpbqfiso}
\text{Cl}(\Z^2 \otimes \text{Sym}^2 \Z^2;D) \cong \text{Cl}^{+}(S(D)).
\end{equation*}

Let $F = (F_1,F_2) \in \Z^2 \otimes \text{Sym}^2 \Z^2$ and let $\x=(x_1,x_2)^t,\y=(y_1,y_2)^t \in \Z^2$. Then we denote $F(\x,\y)$ to be 
$$F(\x,\y) = x_1 F_1(y_1,y_2) + x_2F_2(y_1,y_2).$$

\subsection{Composition identities for pairs of binary quadratic forms}

Let $F = (F_1,F_2)$ and $G = (G_1,G_2)$ be two pairs of binary quadratic forms in $\Z^2 \otimes \text{Sym}^2 \Z^2$ of common discriminant $D$. We define their form product $F \star G$ to be
\begin{align*}
    (F \star G)((\x,\y);(\uu,\vv)) = F(\x,\y) G'(\uu,\vv) + F'(\x,\y) G(\uu,\vv) + \epsilon F(\x,\y) G(\uu,\vv),
\end{align*}
for $\x,\y,\uu,\vv \in \Z^2$
We have the following theorem on composition of pairs of binary quadratic forms.  

\begin{theorem}
\label{thm:pbqfcomp}
    Let $F = (F_1,F_2), G = (G_1,G_2), H=(H_1,H_2) \in \Z^2 \otimes \mathrm{Sym}^2 \Z^2$ be three projective pairs of binary quadratic forms of common discriminant $D$ and $\epsilon \in \{0,1 \}$ congruent to $D$ mod $4$. Let $\x,\y,\uu,\vv \in \Z^2$. Then 
    $$[F] + [G] + [H] = [\mathrm{id}]$$
    if and only if there exist two cubes $R$ and $S$ of discriminant $D$ satisfying
    \begin{align*}
        (G \star H)((\x,\y);(\uu,\vv)) = F(R^\sigma(\x,\uu),S(\y,\vv))
    \end{align*}
    and 
    \begin{align*}
        Q_1^{R} = Q_1^F, \ Q_2^{R} = Q_1^G,
    \end{align*}
    and 
    \begin{equation*}
        \begin{split}
            Q_1^G(1,0) Q_1^H(1,0) = Q_1^F(r_{211},r_{111}), \\
            Q_2^G(1,0) Q_2^H(1,0) = Q_2^F(s_{211},s_{111}).
        \end{split}
    \end{equation*}
\end{theorem}

\begin{proof}

We omit the proof which is similar to that of Theorem \ref{thm:cubecomp}.

\end{proof}

\subsection{Example}

We consider the three elements $F=(F_1,F_2),G= (G_1,G_2), H=(H_1,H_2) \in \Z^2 \otimes \text{Sym}^2 \Z^2$ given by
\begin{equation*}
    \begin{split}
        (F_1(x,y),F_2(x,y) &= (80xy-63y^2,x^2-30xy+23y^2), \\
        (G_1(x,y),G_2(x,y)) &= (-9x^2-2xy+y^2,10x^2+4xy-y^2), \\
        (H_1(x,y),H_2(x,y)) &= (2x^2-y^2,-x^2-4xy+y^2).
    \end{split}
\end{equation*}
Then, $F,G$ and $H$ have discriminant $-31$ and satisfy
$$[F]+[G]+[H] = [\text{id}]$$ 
in $\text{Cl}(\Z^2 \otimes \text{Sym}^2 \Z^2;-31)$. The corresponding companion pairs are given by
\begin{equation*}
    \begin{split}
        (F_1'(x,y),F_2'(x,y) &= (1600x^2 - 2560xy+1016y^2,-569x^2+910xy-361y^2), \\
        (G_1'(x,y),G_2'(x,y)) &= (x^2+18xy+y^2,6x^2-20xy-2y^2), \\
        (H_1'(x,y),H_2'(x,y)) &= (-8xy+y^2,-8x^2+8xy+3y^2).
    \end{split}
\end{equation*}

We have the composition identity 
\begin{equation*}
    (G \star H)((\x,\y);(\uu,\vv)) = F(R^\sigma(\x,\uu),S^\sigma(\y,\vv)),
\end{equation*}
where the cubes $R$ and $S$ are given by
\begin{equation*}
    R = \cube{20}{70}{-6}{(-101)}{-7}{-25}{2}{36}{\ \ \ \ \  ,}\  S = \cube{-4}{9}{4}{1}{4}{-7}{-3}{(-1)}{.}
\end{equation*}

\section{Composition of pairs of quaternary alternating 2-forms} \label{sec:pqaltcomp}

The last two composition laws are defined on spaces obtained from the space of $2 \times 2 \times 2$ cubes by skew-symmetrization operations. The first is the space $\Z^2 \otimes \wedge^2 \Z^4$ of pairs of quaternary alternating $2$-forms obtained by applying a double skew-symmetry operation on cubes. In this case, there is a natural map 
$$ \phi: \Z^2 \otimes \Z^2 \otimes \Z^2 \rightarrow \Z^2 \otimes \wedge^2 \Z^4$$
defined in Section 2.6 of \cite{bhargava:2004}.  To describe this map we identify $\Z^2$ with its $\Z$-dual and let $\x,\yo,\yt,\zo.\zt\in \Z^2.$  For a cube $A\in \Z^2\otimes\Z^2\otimes\Z^2$ define
\begin{equation}
  \label{eq:1}
  \phi(A)(\x, (\yo,\yt), (\zo,\zt))
=A(\x,\yo,\zt)-A(\x,\zo,\yt).
\end{equation}
 After making a choice of basis, this map is explicitly given by
\begin{equation}
\label{eq:cubetoaltpair}
\cube{a}{b}{c}{d}{e}{f}{g}{h}{} \rightarrow \left( \begin{bmatrix}
& & a & b \\
& & c & d \\
-a & -c & & \\
-b & -d & &
\end{bmatrix}, \begin{bmatrix}
& & e & f \\
& & g & h \\
-e & -g & & \\
-f & -h & &
\end{bmatrix} \right).
\end{equation} 

The group $\text{SL}_2(\Z) \times \text{SL}_4(\Z)$ has a natural action on $\Z^2 \otimes \wedge^2 \Z^4$ and this action has a unique polynomial invariant, again called the discriminant of a pair $(F_1,F_2) \in \Z^2 \otimes \wedge^2 \Z^4$. 
It is given by $\text{disc}(\text{Pfaffian}(F_1x-F_2y))$, where the Pfaffian of a $4 \times 4$ skew-symmetric matrix $M$ is given by
$$ \text{Pfaffian}(M) = \sqrt{\text{Det}(M)}.$$
The choice of sign in the square-root is such that $\text{Pfaffian}\left( \begin{pmatrix} 0 & I \\ -I & 0 \end{pmatrix} \right) = 1,$ where $I$ is the $2 \times 2$ identity matrix.
A pair $(F_1,F_2) \in \Z^2 \otimes \wedge^2 \Z^4$ is said to be projective if it is $\text{SL}_2(\Z) \times \text{SL}_4(\Z)$-equivalent to a pair which is the image of a projective cube under the mapping $\phi:\Z^2 \otimes \Z^2 \otimes \Z^2 \rightarrow \Z^2 \otimes \wedge^2 \Z^4$ of \eqref{eq:cubetoaltpair}.  The map $\phi$ is clearly not surjective, but Bhargava shows that the map is surjective onto    $\text{SL}_2(\Z) \times \text{SL}_4(\Z)$ equivalence classes (Section 3.6 of \cite{bhargava:2004}).

\subsection{Correspondence between pairs of quaternary alternating 2-forms and pairs of ideals}

To describe the correspondence between the forms and ideals, we need to recall some notation from (\cite{bhargava:2004}). 
Let $S$ be an oriented quadratic ring and let $K=S \otimes \mathbb{Q}$. 

\begin{definition}
A \emph{rank $n$ ideal} of $S$ is an $S$-submodule of $K^n$ having rank $2n$ as a $\Z$-module. An oriented rank $n$ ideal of $S(D)$ is a pair $(M,\epsilon)$, where $M$ is any rank $n$ ideal of $S(D)$ and $\epsilon=\pm 1$ gives the orientation of $M$. 
\end{definition}

The norm of an oriented rank $n$ ideal $(M,\epsilon)$ is defined to be $\epsilon \cdot |L/M| \cdot |L/S^n|^{-1}$, where $L$ denotes any lattice in $K^n$ containing both $S^n$ and $M$. We will call the usual rank 1 (fractional) ideals simply as ideals. The product of two oriented rank $n$ ideals $(M_1,\epsilon_1)$ and $(M_2,\epsilon_2)$ is defined to be $(M_1 M_2,\epsilon_1 \epsilon_2)$, where $M_1M_2$ is the component-wise multiplication in $K^n$. Two rank $n$ ideals $(M_1,\epsilon_1)$ and $(M_2,\epsilon_2)$ are said to be equivalent if there exists $\lambda \in \text{GL}_n(K)$ such that $(M_1,\epsilon_1) = (\lambda M_2, \text{sgn}(N(\text{det}(\lambda))) \epsilon_2) $.

Now, let $M$ be an oriented rank 2 ideal of $S$. There is a canonical map $\mathrm{det}: (K^2)^2 \rightarrow K$ given by
\begin{equation*}
    \mathrm{det}(\underline{X},\underline{Y})=\begin{vmatrix} X_1&X_2\\Y_1&Y_2 \end{vmatrix},
\end{equation*}
where $\underline{X} = (X_1,X_2)$ and $\underline{Y} = (Y_1,Y_2)$ are in $K^2$. Let $\mathrm{Det}(M)$ denote the ideal in $S$ generated by elements of the form $\mathrm{det}(\underline{X},\underline{Y})$ for $\underline{X},\underline{Y} \in M$. 
A pair $(I,M)$, where $I$ and $M$ are oriented ideals of $S$ having ranks 1 and 2, is called \emph{balanced} if $I \cdot \mathrm{Det}(M) \subset S$ and $N(I)N(M)=1$. Two such balanced pairs $(I,M)$ and $(I',M')$ are said to be equivalent if there exist $\lambda_1 \in \mathrm{GL}_1(K)$ and $\lambda_2 \in \mathrm{GL}_2(K)$ such that $I'=\lambda_1 I$ and $M'=\lambda_2 M$.

As proved in (\cite{bhargava:2004}), there is a canonical bijection between the set of nondegenerate $\mathrm{SL}_2(\Z) \times \mathrm{SL}_4(\Z)$-orbits on the space $\Z^2 \otimes \wedge^2 \Z^4$ of pairs of quaternary alternating $2$-forms and the set of isomorphism classes of pairs of $(S,(I,M))$, where $S$ is a nondegenerate oriented quadratic ring and $(I,M)$ is an equivalence class of balanced pairs of ideals of $S$ having ranks 1 and 2 respectively. Under this bijection, the discriminant of a pair of quaternary alternating 2-forms is equal to the discriminant of the corresponding quadratic ring.

Given a pair $(S,(I,M))$, we now describe how to construct a corresponding pair of quaternary alternating 2-forms. Let $\langle 1,\tau \rangle$ be a positively oriented $\Z$-basis of $S$, and let $\langle \alpha_1,\alpha_2 \rangle$ and $\langle \beta_1,\beta_2, \beta_3, \beta_4 \rangle$ be appropriately oriented $\Z$-bases of $I$ and $M$ respectively. Since $(I,M)$ is balanced, we can write
\begin{equation}
\label{eq:alt2formssys}
    \alpha_i \mathrm{det}(\beta_j,\beta_k) = c^{(i)}_{jk}+a^{(i)}_{jk}\tau,
\end{equation}
for some $c^{(i)}_{jk}, a^{(i)}_{jk} \in \Z$ satisfying 
$$ c^{(i)}_{jk}=-c^{(i)}_{kj}, \ a^{(i)}_{jk}=-a^{(i)}_{kj},$$ for $i=1,2$ and $j,k=1,2,3,4$. The pair of forms in $\Z^2 \otimes \wedge^2 \Z^4$ corresponding to $(I,M)$ is given by 
\begin{equation*} \label{eq:pairquataltforms}
(F_1, F_2) = \left( \sum_{j,k=1}^{4} a_{jk}^{(1)} e_j \wedge e_k, \sum_{j,k=1}^{4} a_{jk}^{(2)} e_j \wedge e_k \right),
\end{equation*}
where $\langle e_1,e_2,e_3,e_4 \rangle$ is the standard basis of $\Z^4$. The pair $(F_1',F_2') \in \Z^2 \otimes \wedge^2 \Z^4$ given by 
\begin{equation*} \label{eq:companion_pairquataltforms} 
(F_1', F_2') = \left( \sum_{j,k=1}^{4} c_{jk}^{(1)} e_j \wedge e_k, \sum_{j,k=1}^{4} c_{jk}^{(2)} e_j \wedge e_k \right)
\end{equation*}
is called the \emph{companion pair} of $(F_1,F_2)$.

Let $\x=(x_1,x_2)^t \in \Z^2$ and  $\y=(y_1,y_2,y_3,y_4)^t,\z=(z_1,z_2,z_3,z_4)^t \in \Z^4$. Using equation (\ref{eq:alt2formssys}), we can write 
\begin{equation}\label{eq:altforms_ideals_relation}
     (x_1\alpha_1+x_2\alpha_2)\cdot \text{det}(\Y,\ZZ) = [x_1F_1(\y,\z) +x_2F_2(\y,\z)] \tau + [x_1F_1'(\y,\z) +x_2F_2'(\y,\z)],
\end{equation}
where $\Y = \sum_{j=1}^4 y_j \beta_j, \ZZ = \sum_{j=1}^4 z_j \beta_j$. As a shorthand, we will write
\begin{equation*}
    F(\x,\y,\z) = x_1F_1(\y,\z) +x_2F_2(\y,\z).
\end{equation*}

A pair of forms $(F_1,F_2) \in \Z^2 \otimes \wedge^2 \Z^4$ is \emph{projective} if and only if the ideals in the corresponding pair $(I,M)$ are projective $S$-modules. 
Using results of Bass (\cite{bass}) and Serre (\cite{Serre1957-1958}), Bhargava shows that any projective balanced pair of ideals $(I,M)$ is equivalent to a pair of the form $(I,S \oplus I^{-1})$. Thus, the set of equivalence classes of projective balanced pair of ideals $(I,M)$ of a quadratic ring $S$ is endowed with an abelian group structure given by
$$ [(I_1,S \oplus I_1^{-1})] * [(I_2, S \oplus I_2^{-1})] = [(I_1 I_2, S \oplus (I_1 I_2)^{-1})].$$
Hence, the projective orbits of $\Z^2 \otimes \wedge^2 \Z^4$ of a fixed discriminant $D$ under $\text{SL}_2(\Z) \times \text{SL}_4(\Z)$ action have a natural group structure via the correspondence with balanced pairs $(I,M)$. This group is denoted as $\text{Cl}(\Z^2 \otimes \wedge^2 \Z^4; D)$, and we have the following isomorphism:
$$ \text{Cl}(\Z^2 \otimes \wedge^2 \Z^4;D) \cong \text{Cl}^+(S(D)).$$
This follows from the fact that any projective balanced pair $(I,M)$ is equivalent to one of the form $(I,S \oplus I^{-1}).$

Similar to Proposition \ref{prop:companioncube}, we have the following proposition which provides the relation between a pair of forms $(F_1,F_2)$ and its companion pair $(F_1',F_2').$
First we note that the $S$-module structure of $I$ is given by equation (\ref{eq:SmoduleI1}) for $p_1,q_1,r_1$ defined by 
\begin{equation*}
    \text{Pfaffian}(F_1x-F_2y) = p_1x^2 + q_1xy+r_1 y^2.
\end{equation*}
In the proposition below we set $\x'=(x_1',x_2')^t$ as in equation (\ref{eq:xprime}).

\begin{proposition} \label{prop:companion_alt2forms}
Let $F = (F_1,F_2) \in \Z^2 \otimes \wedge^2 \Z^4$ be a pair of quaternary alternating $2$-forms and let $(F_1',F_2')$ be its companion pair of forms. We have 
\begin{equation*} \label{eq:altforms_companion_relation}
    F'(\x,\y,\z) = F(\x',\y,\z) - \epsilon F(\x,\y,\z).
\end{equation*}
\end{proposition}

\begin{proof}
Multiplying both sides of equation (\ref{eq:altforms_ideals_relation}) by $\tau$, we get
\begin{equation*}
    \tau (x_1\alpha_1+x_2\alpha_2)\cdot \text{det}(\Y,\ZZ) = F(\x,\y,\z)\tau^2 + F'(\x,\y,\z)\tau.
\end{equation*}
Thus
\begin{equation*}
    (x_1'\alpha_1+x_2'\alpha_2) \cdot \text{det}(\Y,\ZZ) = F(\x,\y,\z)\tau^2 + F'(\x,\y,\z)\tau.
\end{equation*}
Replacing $\x$ by $\x'$ in (\ref{eq:altforms_ideals_relation}), we also have 
\begin{equation*}
    (x_1'\alpha_1+x_2'\alpha_2)\cdot \text{det}(\Y,\ZZ) = F(\x',\y,\z) \tau + F'(\x',\y,\z).   
\end{equation*}
Combining the last two equations gives us
\begin{equation*}
    F(\x,\y,\z)\tau^2 + F'(\x,\y,\z)\tau = F(\x',\y,\z) \tau + F'(\x',\y,\z). 
\end{equation*}
Finally, using equation (\ref{eq:taurelation}) and equating the coefficients of $\tau$ on both sides of the above equation completes the proof of the Proposition. 
\end{proof}

\subsection{Composition identities for pairs of alternating quaternary forms}

Let $F = (F_1,F_2)$ and $G = (G_1,G_2)$ be two pairs of quaternary alternating $2$-forms in $\Z^2 \otimes \wedge^2 \Z^4$ of common discriminant $D$ and $\epsilon\in\{0,1\}$ congruent to $D$ mod 4. We define their form product $F \star G$ to be
\begin{align*}
    (F \star G)((\x,\y,\z);(\uu,\vv,\ww)) =  F(\x,&\y,\z)  G'(\uu,\vv,\ww) \\ & + F'(\x,\y,\z) G(\uu,\vv,\ww) + \epsilon F(\x,\y,\z) G(\uu,\vv,\ww),
\end{align*}
for $\x,\uu \in \Z^2$ and $\y,\z,\vv,\ww \in \Z^4$. 

Before we state our main result on composition of pairs of quaternary alternating $2$-forms, we need the following proposition. We recall that the map $\phi$ was defined by equation \eqref{eq:cubetoaltpair}.

\begin{proposition} \label{prop:doublysymcube}
    Let $F=(F_1,F_2),G=(G_1,G_2),H=(H_1,H_2)$ be three projective quaternary alternating $2$-forms of a common discriminant $D$ satisfying
    $$[F]+[G]+[H] = [\mathrm{id}]$$
in $\mathrm{Cl}(\Z^2 \otimes \wedge^2 \Z^4;D)$. Then there exist cubes $A,B,C$ such that one of them is a doubly symmetric cube, and $\phi(A),\phi(B)$ and $\phi(C)$ are $\mathrm{SL}_2(\Z) \times \mathrm{SL}_4(\Z)$-equivalent to $F,G$ and $H$ respectively.
\end{proposition}

\begin{proof}

Let $S=S(D)$ be the oriented quadratic ring of discriminant $D$. A result of Bass (\cite{bass}) implies that a balanced pair $(I,M)$, where $I$ and $M$ are ideals of $S$ having ranks $1$ and $2$ respectively, is equivalent to a balanced pair of the form $(I,J \oplus K)$ where $I,J,K$ are fractional $S$-ideals. Hence, we may assume without loss of generality that there exist balanced pairs $(I_1,I_2 \oplus I_3),(J_1,J_2 \oplus J_3)$ and $(K_1,K_2 \oplus K_3)$ corresponding to $F,G,$ and $H$ respectively. Also, $I_1J_1K_1=S$ because $[F] + [G] + [H] = [\text{id}]$.

We also recall that, by the cancellation theorem of Serre, a projective module of rank $2$ is determined by its determinant. Hence, any two balanced pair $(I,M_1)$ and $(I,M_2)$ are equivalent because $\text{Det}(M_1)=\text{Det}(M_2) = I^{-1}$.

Let $Cl = \text{Cl}(S(D))$ denote the (usual) class group of the quadratic ring $S(D)$. Let $Cl^2$ denote the subgroup of squares in $Cl$. Then, we have the following two possibilities for the unoriented ideal classes of $I_1,J_1$ and $K_1$.

\textbf{Case 1:} The ideal classes of one of $I_1$ and $J_1$ is in $Cl^2$.

\textbf{Case 2:} The ideal classes of neither $I_1$ and $J_1$ is in $Cl^2$. In this case, the ideal class of $K_1$ is in $Cl^2$, because $Cl^2$ is an index two subgroup of $Cl$.

Hence, one of the ideals $I_1,J_1$ or $K_1$ is equivalent to a square of an ideal. Let us say, without loss of generality, that $I_1$ is equivalent to $L_1^2$ for some ideal $L_1$. Then, the balanced pairs $(I_1,I_2 \oplus I_3)$ and $(L_1^2,L_1^{-1} \oplus L_1^{-1})$ are equivalent by Serre's theorem. Hence, $F,G,H$ are equivalent to $\phi(A),\phi(B),\phi(C)$, where $A,B$ and $C$ are the cubes corresponding to the triples $(L_1^2,L_1^{-1}, L_1^{-1}), (J_1,J_2,J_3)$ and $(K_1,K_2,K_3)$ respectively. 


\end{proof}

Suppose $[F]+[G]+[H]=[\id]$ in $\text{Cl}(\Z^2 \otimes \wedge^2 \Z^4; D)$.  We say the representative forms $F,G,H$ are well-chosen if
\begin{itemize}
  \item the three forms are in the image of $\phi$: $F=\phi(A), G=\phi(B), H=\phi(C)$ for some cubes $A,B,C\in \Z^2\otimes\Z^2\otimes\Z^2$,
  \item $[A]+[B]+[C]=[\id]$ in $\text{Cl}(\Z^2 \otimes \Z^2\otimes\Z^2; D)$ and
  \item $A$ is doubly-symmetric.
\end{itemize}
It is easy to see from Proposition \ref{prop:doublysymcube} that given the relation $[F]+[G]+[H]=[\id]$, we can always replace $F,G,H$ by well-chosen  $\text{SL}_2(\Z) \times \text{SL}_4(\Z)$ equivalent representatives, after possibly permuting $F,G,H$.

\begin{theorem}\label{thm:quaternary}
    Let $F, G$ and  $H$ be three projective pairs of quaternary alternating $2$-forms of common discriminant $D$ and $\epsilon\in\{0,1\}$ congruent to $D$ mod 4.
 Let $\x,\underline{y_i}, \underline{z_i},\uu ,\underline{v_i}, \underline{w_i} \in \Z^2$ and $\y=(\yo,\yt),\z=(\zo,\zt),\vv=(\vo,\vt),\ww=(\wo,\wt) \in \Z^4.$
    Then
    $$[F]+[G]+[H]=[\id],$$
with $F,G,H$ well-chosen representatives (in the images of cubes $A,B,C$ respectively)
    if and only if there exist cubes $R=(r_{ijk}), S=(s_{ijk}), T=(t_{ijk})$ of discriminant $D$ satisfying
    \begin{align} \label{eq:alt2formscomp}
    (G \star H)&(\x,\y,\z;\uu,\vv,\ww)\\ \nn
    = F&(R^\sigma(\x,\uu), (S^\sigma(\yo,\vo) + T^\sigma(\yt,\vt), S^\sigma(\yo,\wo) + T^\sigma(\yt,\wt)), \\ \nn  & (S^\sigma(\zo,\vo) + T^\sigma(\zt,\vt), S^\sigma(\zo,\wo) + T^\sigma(\zt,\wt))),
    \end{align}
    with
    \begin{align}\label{eq:Q1R}
    Q_1^{R} = Q_1^A, \ Q_2^{R} = Q_1^B,
    \end{align}
    and 
    \begin{align}\label{eq:QABC}
      Q_1^B(1,0)\cdot Q_1^C(1,0)&=Q_1^A(r_{211}, r_{111}),\\
  \nn    Q_2^B(1,0)\cdot Q_2^C(1,0)&=Q_2^A(s_{211}, s_{111}),\\
   \nn   Q_3^B(1,0)\cdot Q_3^C(1,0)&=Q_3 ^A(t_{211}, t_{111}).
    \end{align}
\end{theorem}

\begin{proof}
  We start by assuming that $[F]+[G]+[H]=[\id]$ with $F=\phi(A), G=\phi(B), H=\phi(C)$ with $A$ doubly-symmetric and $[A]+[B]+[C]=[\id].$  By Theorem \ref{thm:cubecomp}, this last condition implies that there exist cubes $R,S,T$ such that
  \begin{equation}
    \label{eq:alternating-pairs-proof}
  (B\star C)(\x,\underline{y_i},\underline{z_j};\uu,\underline{v_k},\underline{w_l})
    = A(R^\sigma(\x,\uu), S^\sigma(\underline{y_i},\underline{v_k}), T^\sigma(\underline{z_j},\underline{w_l}))
  \end{equation}
  and the normalizing conditions  \eqref{eq:Q1R}, \eqref{eq:QABC} are satisfied. 
  
Proposition \ref{prop:companion_alt2forms} implies $\phi(B)' = \phi(B')$ and $\phi(C)'=\phi(C')$.  Hence
\begin{align*}
  (G \star H)&(\x,\y,\z;\uu,\vv,\ww)\\
&= G(\x,\y,\z)H^\prime(\uu,\vv,\ww)+ G^\prime(\x,\y,\z)H(\uu,\vv,\ww)+ \epsilon G(\x,\y,\z)H(\uu,\vv,\ww) \\
&= \phi(B)(\x,\y,\z)\phi(C^\prime)(\uu,\vv,\ww)+ \phi(B^\prime)(\x,\y,\z)\phi(C)(\uu,\vv,\ww) \\
& \ \ \ \ \ \ \ \ + \epsilon \phi(B)(\x,\y,\z)\phi(C)(\uu,\vv,\ww) \\
\end{align*}
Using the definition \eqref{eq:1} of $\phi$ together with \eqref{eq:alternating-pairs-proof}, a little algebra now yields  \eqref{eq:alt2formscomp}.

Conversely, with $F,G,H \in \Z^2 \otimes \wedge^2 \Z^4$ and $A,B,C \in \Z^2 \otimes \Z^2 \otimes \Z^2$ be as in the statement of the theorem,  assume that 
\begin{align} 
    (G \star H)&(\x,\y,\z;\uu,\vv,\ww)\\ \nn
    = F&(R^\sigma(\x,\uu), (S^\sigma(\yo,\vo) + T^\sigma(\yt,\vt), S^\sigma(\yo,\wo) + T^\sigma(\yt,\wt)), \\ \nn  & (S^\sigma(\zo,\vo) + T^\sigma(\zt,\vt), S^\sigma(\zo,\wo) + T^\sigma(\zt,\wt))),
\end{align}
is satisfied for some cubes $R,S,T$ of discriminant $D$. We also assume \eqref{eq:Q1R}, \eqref{eq:QABC}. We will prove that $[A]+[B]+[C]=[\id]$. The relation $[F]+[G]+[H]=[\id]$ will then follow.

Observe that substituting $\yt=\zo=(0,0)$ and $\vt=\wo=(0,0)$ in \eqref{eq:alt2formscomp} gives 
\begin{align*}
    (B \star C)(\x,\yo,\zt;\uu,\vo,\wt) =  A(R^\sigma(\x,\uu),S^\sigma(\yo,\vo),T^\sigma(\zt,\wt)).
\end{align*}

Under the assumptions stated above, by Theorem \ref{thm:cubecomp}, we immediately conclude that $[A]+[B]+[C]=[\id]$. This completes the proof of the theorem.

\end{proof}

\subsection{Example}

We consider the pairs $F = (F_1,F_2), G=(G_1,G_2)$ and $H = (H_1,H_2)$ in $\Z^2 \otimes \wedge^2 \Z^4$ of discriminant $-47$ which are equal to $\phi(A),\phi(B)$ and $\phi(C)$ respectively for the cubes
\begin{equation*} 
A = \cube{0}{2}{2}{-1}{1}{0}{0}{-3}, 
B = \cube{0}{-1}{-2}{-1}{-1}{0}{0}{6},
C = \cube{-1}{2}{2}{-2}{1}{5}{-1}{-11}.
\end{equation*}
We have 
$$[A] + [B] + [C] = [\id]$$
in $\text{Cl}(\Z^2 \otimes \Z^2 \otimes \Z^2;-47)$, and the cube $A$ is doubly-symmetric. Also, the associated quadratic forms are given by $Q_1^A=[4,-1,3], Q_2^A=[2,1,6],Q_3^A=[2,1,6], Q_1^B = [2,1,6], Q_2^B=[2,-1,6], Q_3^B=[1,-1,12], Q_1^C = [2,1,6], Q_2^C = [1,1,12]$ and $Q_3^C = [7,25,24].$

The companion pairs $F',G'$ and $H'$ are given by $\phi(A'),\phi(B')$ and $\phi(C')$ respectively, where $A',B',C'$ are the companion cubes of $A,B,C$ respectively and are equal to
\begin{equation*} 
A' = \cube{4}{-2}{-2}{-11}{0}{-6}{-6}{3}, 
B' = \cube{-2}{0}{0}{12}{1}{6}{12}{0},
C' = \cube{2}{10}{-2}{-22}{5}{-17}{-11}{23}.
\end{equation*}

The cubes $(R,S,T)$ which are dual to $(A,B,C)$ are given by 
\begin{equation*} 
R = \cube{0}{2}{2}{-1}{1}{0}{0}{-3}, 
S = \cube{1}{-2}{-1}{0}{-1}{1}{-6}{12},
T = \cube{0}{2}{1}{0}{1}{-1}{0}{-6},
\end{equation*}
and the pair of forms $F,G,H$ satisfy the composition identity 
\begin{align*} 
    (G \star H)&(\x,\y,\z;\uu,\vv,\ww)\\ \nn
    = F&(R^\sigma(\x,\uu), (S^\sigma(\yo,\vo) + T^\sigma(\yt,\vt), S^\sigma(\yo,\wo) + T^\sigma(\yt,\wt)), \\ \nn  & (S^\sigma(\zo,\vo) + T^\sigma(\zt,\vt), S^\sigma(\zo,\wo) + T^\sigma(\zt,\wt))),
\end{align*}

\section{Composition of senary alternating 3-forms} \label{sec:senaltcomp}

The final composition law we consider is the composition of senary alternating $3$-forms. The space $\wedge^3 \Z^6$ of senary alternating $3$-forms is obtained by imposing triple skew-symmetry on the space $\Z^2 \otimes \Z^2 \otimes \Z^2$. As in the previous section, we have a natural map
\begin{equation} \label{eq:cubetosenaryaltmap}
\wedge_{2,2,2} : \Z^2 \otimes \Z^2 \otimes \Z^2 \rightarrow \wedge^3(\Z^2 \oplus \Z^2 \oplus \Z^2)=\wedge^3 \Z^6,
\end{equation}
(See Section 2.7 of \cite{bhargava:2004} for the definition of this map.)

The group $\text{SL}_6(\Z)$ acts on the space $\wedge^3 \Z^6$  and there is a unique polynomial invariant under this action called the discriminant. A form $E \in \wedge^3 \Z^6$ is called \emph{nondegenerate} if its discriminant is nonzero and it is called \emph{projective} if it is equivalent to a form which is in the image of the mapping $\wedge_{2,2,2}$. 

\subsection{Correspondence between senary alternating 3-forms and rank 3 ideals}

Let $S$ be an oriented quadratic ring, $K=S \otimes \Q$, and let $M \subset K^3$ be a rank 3 ideal of $S$. There is a canonical map $\mathrm{det}: (K^3)^3 \rightarrow K$, given by
\begin{equation*}
    \mathrm{det}(\underline{X},\underline{Y},\underline{Z})=\begin{vmatrix}X_1&X_2&X_3\\Y_1&Y_2&Y_3\\Z_1&Z_2&Z_3\end{vmatrix}.
\end{equation*}
Let $\mathrm{Det}(M)$ be the ideal generated by elements of the form $\mathrm{det}(\underline{X},\underline{Y},\underline{Z})$, where $\underline{X},\underline{Y},\underline{Z} \in M$. If $M$ is a decomposable ideal i.e. $M = I_1 \oplus I_2 \oplus I_3 \subset K^3$, for some rank $1$ ideals $I_1,I_2,I_3$ of $S$, then $\mathrm{Det}(M)=I_1I_2I_3$. We say $M$ is balanced if $\mathrm{Det}(M) \subset S$ and $N(M)=1$. Two rank $3$ ideals $M_1$ and $M_2$ are said to be equivalent if there exists $\lambda \in GL_3(K)$ such that $\lambda \cdot M_1 = M_2$.

Bhargava (\cite{bhargava:2004}) proves there is a canonical bijection between the set of nondegenerate $\mathrm{SL}_6(\Z)$-orbits on the space $\wedge^3 \Z^6$ of senary alternating $3$-forms, and the set of isomorphism classes of pairs of $(S,M)$, where $S$ is a nondegenerate oriented quadratic ring and $M$ is an equivalence class of balanced rank 3 ideals of $S$. Under this bijection, the discriminant of a senary alternating 3-form is equal to the discriminant of the corresponding quadratic ring.

To construct the senary alternating 3-form corresponding to a given pair $(S,M)$, let $\langle 1,\tau \rangle$ be a positively oriented $\Z$-basis of $S$ and let $\mathcal{A} = \langle \alpha_1,\alpha_2,\alpha_3,\alpha_4,\alpha_5,\alpha_6 \rangle$ be a positively oriented $\Z$-basis of $M$. Since $M$ is balanced, we have 
\begin{equation}
\label{eq:senarysys}
    \mathrm{det}(\alpha_i,\alpha_j,\alpha_k) = a'_{ijk}+a_{ijk}\tau,
\end{equation}
for some integers $a_{ijk}$ and $a'_{ijk}$ satisfying 
\begin{equation*}
    a_{ijk} = -a_{ikj} = -a_{jik} = -a_{kij},
\end{equation*}
and
\begin{equation*}
    a'_{ijk} = -a'_{ikj} = -a'_{jik} = -a'_{kij},    
\end{equation*}
for $i,j,k \in \lbrace 1,2,3,4,5,6 \rbrace$. Then, 
$$E = \sum_{i,j,k = 1}^{6} a_{ijk} e_i \wedge e_j \wedge e_k$$ is the senary alternating 3-form corresponding to $M$, where $\langle e_1,e_2,...,e_6 \rangle$ is the standard basis of $\Z^6$. The form 
\begin{equation*} \label{eq:companion_senaryaltform}
E' = \sum_{i,j,k = 1}^{6} a'_{ijk} e_i \wedge e_j \wedge e_k
\end{equation*}
is called the \emph{companion form} of $E$. Let $\underline{X}, \underline{Y}$ and $\underline{Z}$ be arbitrary vectors in $M$ and let $\underline{x} = (x_1,x_2,...,x_6)^t, \underline{y} = (y_1,y_2,...,y_6)^t$ and $\underline{z} = (z_1,z_2,...,z_6)^t$ be their respective $\Z^6$ coordinates with respect to the $\Z$-basis $\mathcal{A}$ of $M$. Using \eqref{eq:senarysys}, we can write 
\begin{equation}
    \mathrm{det}(\underline{X},\underline{Y},\underline{Z}) = E(\underline{x},\underline{y},\underline{z})\tau + E'(\underline{x},\underline{y},\underline{z}).
\end{equation}

Again using results of Bass (\cite{bass}) and Serre (\cite{Serre1957-1958}) we can deduce that any rank 3 ideal $M$ of $S$ is (isomorphic to) a decomposable ideal i.e. the direct sum of rank 1 ideals, and that any projective balanced rank 3 ideal of $S$ is equivalent to $S \oplus S \oplus S$. This implies that the class group of senary alternating $3$-forms of any given discriminant $D$ is the trivial group. The identity element of $\text{Cl}(\wedge^3 \Z^6;D)$ is the class of $E_{\id,D}$, where 
$$E_{\id,D} = \wedge_{2,2,2}(A_{\id,D}).$$

\subsection{Composition identities for senary alternating 3-forms}

Let $E_1$ and $E_2$ be two senary alternating $3$-forms in $\wedge^3 \Z^6$ of common discriminant $D$. We define their form product $E_1 \star E_2$ to be
\begin{align*}
    (E_1 \star E_2)((\x,\y,\z);(\uu,\vv,\ww)) =  E_1&(\x,\y,\z)  E_2'(\uu,\vv,\ww) \\ & + E_1'(\x,\y,\z) E_2(\uu,\vv,\ww) + \epsilon E_1(\x,\y,\z) E_2(\uu,\vv,\ww),
\end{align*}
for $\x,\y,\z,\uu,\vv,\ww \in \Z^6$. We have the following result on composition of senary alternating $3$-forms.

\begin{theorem} \label{thm:senarycomp}

Let $E=E_{\id,D}$ be the projective senary alternating $3$-form of discriminant $D$ that represents the identity element in $\mathrm{Cl}(\wedge^3 \Z^6;D)$ and $\epsilon \in \{0,1\}$ congruent to $D$ mod $4$. For $i=1,2,3$, let 
$\underline{x_i},\underline{y_i},\underline{z_i},\underline{u_i},\underline{v_i},\underline{w_i} \in \Z^2$ and let $\x=(\xo,\xt,\xth),\y=(\yo,\yt,\yth),\z=(\zo,\zt,\zth),\uu=(\uo,\ut,\uth),\vv=(\vo,\vt,\vth),\ww=(\wo,\wt,\wth).$
Then, we have the composition identity
\begin{align*}
    (E \star E)(\x,\y,\z;\uu,\vv,\ww) & = \\ \nn & E \Bigg [  \left( \sum_{i=1}^3 R^\iota(\underline{x_i},\underline{u_i}),\sum_{i=1}^3 R^\iota(\underline{y_i},\underline{u_i}), \sum_{i=1}^3 R^\iota(\underline{z_i},\underline{u_i}) \right)^t, \\ & \left( \sum_{i=1}^3 R^\iota(\underline{x_i},\underline{v_i}),\sum_{i=1}^3 R^\iota(\underline{y_i},\underline{v_i}), \sum_{i=1}^3 R^\iota(\underline{z_i},\underline{v_i}) \right)^t, \\ & \left( \sum_{i=1}^3 R^\iota(\underline{x_i},\underline{w_i}),\sum_{i=1}^3 R^\iota(\underline{y_i},\underline{w_i}), \sum_{i=1}^3 R^\iota(\underline{z_i},\underline{w_i}) \right)^t \Bigg ].
\end{align*}
where $R$ is the cube of discriminant $D$ given by
$$R=\cube{1}{0}{0}{\frac{D-\epsilon}{4}}{0}{1}{1}{\epsilon}.$$

\end{theorem}

\begin{proof}

The senary alternating $3$-form $E=E_{\id,D}$ corresponds to the pair $(S,S \oplus S \oplus S)$. Indeed, let $\underline{X} = (X_1,X_2,X_3), \underline{Y} = (Y_1,Y_2,Y_3)$ and $\underline{Z}=(Z_1,Z_2,Z_3)$ be arbitrary vectors in $S \oplus S \oplus S$.
We fix the $\Z$-basis $\langle 1, \tau \rangle$ of $S$. For $i=1,2,3$, let $\underline{x_i},\underline{y_i},\underline{z_i} \in \Z^2$ be the respective $\Z^2$ coordinates  of $X_i,Y_i,Z_i$ with respect to the fixed $\Z$-basis of $S$. Then, using \eqref{eq:senarysys}, we can write 
\begin{equation*}
    \mathrm{det}(\underline{X},\underline{Y},\underline{Z}) = E(\x,\y,\z)\tau + E'(\x,\y,\z),
\end{equation*}
where $\x=(\xo,\xt,\xth),\y=(\yo,\yt,\yth),\z=(\zo,\zt,\zth)$.

Similarly, let $\underline{U} = (U_1,U_2,U_3),$ $\underline{V} = (V_1,V_2,V_3)$ and $\underline{W} = (W_1,W_2,W_3)$ be arbitrary vectors in $S \oplus S \oplus S$ and $\underline{u_i},\underline{v_i},\underline{w_i} \in \Z^2$ be the respective $\Z^2$ coordinates  of $U_i,V_i,W_i$. Let $\uu=(\uo,\ut,\uth),\vv=(\vo,\vt,\vth),\ww=(\wo,\wt,\wth).$ Now, let $\underline{A}, \underline{B}, \underline{C} \in S \oplus S \oplus S$ be defined by the relation
\begin{equation} \label{eq:3x3matrix}
    \begin{bmatrix}
        X_1 & X_2 & X_3 \\
        Y_1 & Y_2 & Y_3 \\
        Z_1 & Z_2 & Z_3
    \end{bmatrix} 
    \begin{bmatrix}
        U_1 & V_1 & W_1 \\
        U_2 & V_2 & W_2 \\
        U_3 & V_3 & W_3
    \end{bmatrix}
    = \begin{bmatrix}
        A_1 & B_1 & C_1 \\
        A_2 & B_2 & C_2 \\
        A_3 & B_3 & C_3
    \end{bmatrix}.
\end{equation}
Let $\underline{a_i},\underline{b_i},\underline{c_i} \in \Z^2$ be the respective $\Z^2$ coordinates  of $A_i,B_i,C_i$ and let $\aaa=(\underline{a_1},\underline{a_2},\underline{a_3}), \bb = (\underline{b_1},\underline{b_2},\underline{b_3}),\cc=(\underline{c_1},\underline{c_2},\underline{c_3})$. The cube $R$ corresponds to the balanced triple $(S,S,S)$ where the $\Z$-basis is chosen to be $\langle 1,\tau \rangle$ for all the three copies of $S$. We note that for $(x_1,x_2)^t,(u_1,u_2)^t \in \Z^2$,
\begin{align*}
    (x_1+x_2 \tau)(u_1+u_2 \tau) = a_1+a_2 \tau,
\end{align*}
where $(a_1,a_2)^t = R^\iota((x_1,x_2)^t,(u_1u_2)^t)$.
Using this equation and taking determinants in \eqref{eq:3x3matrix} gives us
\begin{align} \label{eq:senarycompid_in_S}
    [E(\x,\y,\z) \tau + E'(\x,\y,\z)] \cdot [E(\uu,\vv,\ww) \tau + E'(\uu,\vv,\ww)] = E(\aaa,\bb,\cc) \tau + E'(\aaa,\bb,\cc),
\end{align}
where $\aaa,\bb,\cc$ are given by 
\begin{equation*} \label{eq:case1bc}
    \begin{split}
        \aaa & = (\underline{a_1},\underline{a_2},\underline{a_3})  =\left( \sum_{i=1}^3 R^\iota(\underline{x_i},\underline{u_i}),\sum_{i=1}^3 R^\iota(\underline{y_i},\underline{u_i}), \sum_{i=1}^3 R^\iota(\underline{z_i},\underline{u_i}) \right)^t, \\ 
        \bb & = (\underline{b_1},\underline{b_2},\underline{b_3})  = \left( \sum_{i=1}^3 R^\iota(\underline{x_i},\underline{v_i}),\sum_{i=1}^3 R^\iota(\underline{y_i},\underline{v_i}), \sum_{i=1}^3 R^\iota(\underline{z_i},\underline{v_i}) \right)^t, \\
        \cc & = (\underline{c_1},\underline{c_2},\underline{c_3}) = \left( \sum_{i=1}^3 R^\iota(\underline{x_i},\underline{w_i}),\sum_{i=1}^3 R^\iota(\underline{y_i},\underline{w_i}), \sum_{i=1}^3 R^\iota(\underline{z_i},\underline{w_i}) \right)^t.
    \end{split}
\end{equation*}

Finally, equating the coefficients of $\tau$ on both sides of equation \eqref{eq:senarycompid_in_S}, we get
\begin{align*}
    (E \star E)(\x,\y,\z;\uu,\vv,\ww) & = \\ \nn & E \Bigg [  \left( \sum_{i=1}^3 R^\iota(\underline{x_i},\underline{u_i}),\sum_{i=1}^3 R^\iota(\underline{y_i},\underline{u_i}), \sum_{i=1}^3 R^\iota(\underline{z_i},\underline{u_i}) \right)^t, \\ & \left( \sum_{i=1}^3 R^\iota(\underline{x_i},\underline{v_i}),\sum_{i=1}^3 R^\iota(\underline{y_i},\underline{v_i}), \sum_{i=1}^3 R^\iota(\underline{z_i},\underline{v_i}) \right)^t, \\ & \left( \sum_{i=1}^3 R^\iota(\underline{x_i},\underline{w_i}),\sum_{i=1}^3 R^\iota(\underline{y_i},\underline{w_i}), \sum_{i=1}^3 R^\iota(\underline{z_i},\underline{w_i}) \right)^t \Bigg ].
\end{align*}

\end{proof}

\section{Conclusion} \label{sec:conc_ack}

We identify some possible future directions and  generalizations of our work. 
\begin{enumerate}
    \item We use the correspondence between forms and ideals in each higher composition law to compute the composition of given forms and to write the explicit composition identities. It would be interesting to find a more direct approach to the composition identities without having to rely on the correspondence with ideals. 
    \item Our results are potentially applicable to questions of representations of integers by homogeneous forms, e.g. binary cubic forms.  Similar identities in different settings are applied to such problems in recent works of Choudhry \cite{choudhry:2022} and Duke \cite{duke2022some}.
    \item Bhargava investigates further higher composition laws in \cite{hcl2}, \cite{hcl3} and \cite{hcl4}.  We expect our methods to extend these settings. More interesting are extensions beyond the case of prehomogeneous vector spaces. A catalog of potential candidates for study appears in the work of Bhargava and Ho \cite{bhargavaho} on coregular spaces.  The papers  \cite{choudhry:2022} and \cite{duke2022some} cited above (dealing with $3\times 3\times 3$ cubes and symmetric $2\times 2\times 2\times 2$ hypercubes, respectively) are examples.
\end{enumerate}


\bibliographystyle{amsalpha}
\bibliography{references}

\end{document}